\documentclass[reqno,a4paper,10pt]{amsart} 

\usepackage[utf8]{inputenc}
\usepackage[english]{babel}
\usepackage{fullpage}
\usepackage{tikz}
\usepackage{tkz-graph}
\usepackage{tikz-cd}

\usetikzlibrary{patterns,snakes}
\usetikzlibrary{matrix,positioning}

\usepackage{amsmath}
\usepackage{amssymb}
\usepackage{amsthm}
\usepackage{thmtools}
\usepackage[colorlinks = true, linkbordercolor = {white}, urlcolor={purple}, linkcolor={purple}, citecolor = {purple}]{hyperref}
\usepackage{cleveref}
\usepackage{amsfonts}
\usepackage{enumerate}
\usepackage[foot]{amsaddr}
\usepackage{mathtools}
\usepackage[normalem]{ulem}
\usepackage[font=footnotesize]{caption} 
\usepackage{pifont}
\usepackage[percent]{overpic}
\usepackage{dsfont}

\newcommand{\tp}{.}
\newcommand{\tc}{,}
\newcommand{\bm}{\mathbf}

\usetikzlibrary{arrows}

\theoremstyle{definition}

\newtheorem{theorem}{Theorem}[section]
\newtheorem{definition}[theorem]{Definition}
\newtheorem{corollary}[theorem]{Corollary}
\newtheorem{prop}[theorem]{Proposition}
\newtheorem{lemma}[theorem]{Lemma}
\newtheorem{example}[theorem]{Example}

\newtheorem{itheorem}{Theorem}[section]

\newtheorem{icorollary}[itheorem]{Corollary}

\theoremstyle{remark}
\newtheorem{remark}[theorem]{Remark}

\renewcommand{\emph}{\textsl}
\renewcommand{\textit}{\textsl}

\renewcommand{\inf}{\mathop{\mathrm{inf}\vphantom{\mathrm{sup}}}}

\newcommand{\Z}{{\mathbb Z}}

\newcommand{\R}{{\mathbb R}}
\newcommand{\N}{{\mathbb N}}

\newcommand{\mc}{\mathcal}

\newcommand{\dd}{\,\mathrm{d}}

\newcommand{\trans}{T}
\newcommand{\desub}{D}
\newcommand{\hg}{h^{\operatorname{g}}}
\newcommand{\Ent}{\operatorname{H}}
\newcommand{\Length}{L}

\numberwithin{equation}{section}

\usepackage{xcolor}
\usepackage{soul}



\begin{document}

\title{A classification of intrinsic ergodicity for recognisable random substitution systems}

\author{P.\,Gohlke$^{\,1}$ and A.~Mitchell$^{\,2}$}

\maketitle

{\centering{\footnotesize
{$^{\,1}$ Faculty of Mathematics and Computer Science, Friedrich Schiller University, 07743 Jena, Germany\\
$^{\,2}$ Department of Mathematical Sciences, Loughborough University, Loughborough, LE11 3TU, UK\\}
}}

\begin{abstract}
We study a class of dynamical systems generated by random substitutions, which contains both intrinsically ergodic systems and instances with several measures of maximal entropy. In this class, we show that the measures of maximal entropy are classified by invariance under an appropriate symmetry relation. All measures of maximal entropy are fully supported and they are generally not Gibbs measures. We prove that there is a unique measure of maximal entropy if and only if an associated Markov chain is ergodic in inverse time. This Markov chain has finitely many states and all transition matrices are explicitly computable. Thereby, we obtain several sufficient conditions for intrinsic ergodicity that are easy to verify. A practical way to compute the topological entropy in terms of inflation words is extended from previous work to a more general geometric setting.
\end{abstract}

\section{Introduction}

A \emph{substitution} is a symbolic rule that replaces each letter from a finite alphabet with a word consisting of letters from the same alphabet.
Dynamical systems arising from substitutions provide the prototypical examples of mathematical quasicrystals.
Such systems are well-known to have zero topological entropy and, under a mild condition, support a unique invariant measure.
For a detailed introduction to the statistical properties of substitution dynamical systems, we refer the reader to \cite{baake-grimm,queffelec}.

\emph{Random substitutions} are a generalisation of substitutions where the substituted image of a letter is determined by a Markov process.
In contrast to their deterministic counterparts, their associated dynamical systems typically have positive topological entropy \cite{mitchell} and support uncountably many ergodic measures \cite{gohlke-spindeler}.
Random substitution systems thus provide theoretical models for quasicrystals with \emph{local defects}, which simultaneously exhibit long-range order alongside local disorder.
They also provide a systematic framework to study properties of statistically self-similar structures as introduced by Mandelbrot \cite{mandelbrot,mandelbrot2}. 
In particular, Peyri\`{e}re studied the convergence of pattern frequencies \cite{peyriere}, Godr\`{e}che and Luck observed mixed spectral types in the diffraction image \cite{godreche-luck}, and Dekking, Grimmett and Meester classified several phases in random Cantor sets, including the occurrence of fractal percolation \cite{dekking-grimmett,dekking-meester,dekking-wal}.


Variational principles play a pivotal role in finding physically meaningful distributions on a given system. In thermodynamic formalism these distributions are given by equilibrium measures, which often satisfy certain regularity properties \cite{Bowen,Ruelle}. In the absence of an external potential, these equilibrium states coincide with the \emph{measures of maximal entropy} (MMEs); compare also \cite{GS,Jaynes} for the role of entropy maximisation. Within some classes of dynamical systems, it has been observed that both uniqueness of the MME and the presence of several MMEs are possible, depending on the parameters, thus leading to the occurrence of a phase transition \cite{BurtonSteif,Haydn}.

The question of whether a dynamical system has a unique measure of maximal entropy has been well studied over the last decades for various types of dynamical systems \cite{Bufetov_Gurevich, BFSV, Gelfert_Ruggiero, Hofbauer, Markley_Paul, Ures}. Still, there is no complete characterisation to date, even for symbolic dynamical systems. 
Many important classes such as topologically transitive subshifts of finite type have been shown to be \emph{intrinsically ergodic}, that is, there exists a unique measure of maximal entropy \cite{parry_64,Weiss_1,Weiss_2}. 
A common technique for proving a given subshift is intrinsically ergodic is to verify the \emph{specification property} \cite{Bowen74}.
However, there exist many examples of intrinsically ergodic subshifts that do not have specification, some of which are covered by appropriate generalisations of the specification property \cite{Climenhaga_Pavlov, THOMPSON_3, THOMPSON_2, MT-entropy}. Conversely, some weaker versions of specification have been shown to be compatible with several measures of maximal entropy \cite{KOR, Pavlov_counter}. 
A classical example of a non-intrinsically ergodic subshift is the Dyck shift, studied by Krieger \cite{krieger}, which has two fully supported ergodic measures of maximal entropy which are Bernoulli. On the other hand, Haydn produced examples with several measures of maximal entropy that have disjoint topological support \cite{Haydn}. In fact, there are subshifts with uncountably many measures of maximal entropy \cite{Buzzi}.
Progress on the classification of intrinsic ergodicity has also been made recently in the context of coded systems \cite{Pavlov_coded}, suspensions over shifts of finite type \cite{Iommi_Velozo, Kucherenko_Thompson}, and bounded density shifts \cite{GPR}. We also refer to \cite{Climenhaga_Thompson_notes,Kucherenko_Thompson} for more on the history of this problem.

In this work, we classify intrinsic ergodicity for \emph{primitive random substitution systems} under appropriate regularity assumptions.
We show that for this class, the problem of intrinsic ergodicity is non-trivial.
That is, there exist both intrinsically ergodic and non-intrinsically ergodic examples. 
All the measures of maximal entropy have full topological support, but are in general not Bernoulli. In fact, it was shown in previous work that they generally violate a (weak) Gibbs property for the zero potential, and in particular that the corresponding subshift does not satisfy specification \cite{MT-entropy}.
Primitive random substitutions produce systems with complex dynamical properties, including mixed spectral types \cite{baake-spindeler-strungaru, moll}, positive entropy \cite{gohlke,MT-entropy, mitchell}, a hierarchical structure \cite{baake-spindeler-strungaru}, rich automorphism groups \cite{fokkink-rust}, non-trivial dimension spectra \cite{Mitchell_Rutar}, and subtle mixing properties \cite{RMMT,Rust_Miro_Sadun}. With the results presented in this work we therefore contribute to the study of intrinsic ergodicity in a regime of intricate dynamical behaviour. 

A random substitution is given by a set-valued function $\vartheta$ that maps letters from an alphabet $\mc A$ to sets of words in this alphabet. As an example, consider $\mc A= \{ a,b\}$ and $\vartheta \colon a \mapsto \{ aba\}, \, b\mapsto \{baa,bba \}$. It is extended to words by concatenating all possible realisations on the individual letters. For instance, in the given example, $\vartheta(ab) = \{ababaa, ababba \}$. A subshift $X_\vartheta \subset \mc A^\Z$ is assigned in the standard way, by imposing that every pattern in $x \in X_\vartheta$ can be generated from an iteration of $\vartheta$ on some letter. The standard assumption of \emph{primitivity} ensures that $X_\vartheta$ is topologically transitive under the shift map.

The class of primitive random substitutions is very large and encompasses subshifts with contrasting dynamical behaviour, including all topologically transitive shifts of finite type \cite{gohlke-rust-spindeler} and deterministic primitive substitution subshifts, as well as the Dyck shift and similar examples of coded shifts \cite{gohlke-spindeler}. It is therefore customary to either study isolated examples or to impose further assumptions on the class of random substitutions under consideration. We work with an assumption called \emph{geometrical compatibility} that generalises two common assumptions in previous work, \emph{constant length} and \emph{compatibility}. This is also the minimal restriction to ensure that $\vartheta$ allows for a geometric interpretation as a \emph{random inflation rule}. Such a geometric setting seems natural as it readily generalises to shapes in higher dimensions~\cite{godreche-luck}. This geometric framework is adequately represented by a \emph{suspension} $Y_\vartheta$ of the subshift $X_\vartheta$. In the special cases of compatible or constant length random substitutions, intrinsic ergodicity of $X_\vartheta$ and $Y_\vartheta$ are equivalent.

An assumption on $\vartheta$ that puts us outside the scope of many of the classical examples for intrinsic ergodicity is \emph{recognisability}. In fact it was shown in \cite{fokkink-rust} that the corresponding subshifts have non-residually finite automorphism groups and therefore exclude, for instance, all mixing subshifts of finite type.
Recognisability means that every $x \in X_\vartheta$ can be decomposed uniquely into \emph{inflation words} in $\cup_{a \in \mc A} \vartheta^n(a)$ for all $n \in \N$. While this property is automatic for primitive substitutions with infinite subshifts \cite{Mosse}, it has to be imposed as an extra condition for their random analogues. 
Recognisability allows us to identify inflation words in $x \in X_\vartheta$, and locally swapping words in $\vartheta^n(a)$ for some fixed $a \in \mc A$ and $n \in \N$ gives another sequence $y \in X_\vartheta$. These symmetry transformations form the so-called \emph{shuffle group}, which is responsible for the automorphism group being non-residually finite \cite{fokkink-rust}. We call a measure that is invariant under the shuffle group a \emph{uniformity measure}. In fact, measures of maximal entropy are known to respect any symmetry of exchangeable words, up to a factor reflecting a potential change of length \cite{Garcia-Ramos_Pavlov, Meyerovitch}. Our first main result is that, assuming geometric compatibility and recognisability, invariance under the shuffle group entirely characterises the measures of maximal entropy. That is, the measures of maximal entropy on $Y_\vartheta$ are \emph{precisely} the uniformity measures. Since uniformity measures have full topological support, the same holds for the measures of maximal entropy.

Our second main result gives a characterisation of the uniqueness of uniformity measures, and hence of the intrinsic ergodicity of $Y_\vartheta$. We harvest the fact that equidistributing the inflation words in $\vartheta^n(a)$ for \emph{all} levels $n$ and $a \in \mc A$ imposes some rigidity on the uniformity measures in the form of self-consistency relations. These are encoded in a sequence of Markov matrices $Q_\vartheta = (Q_n)_{n \in \N}$, whose entries can be written explicitly in terms of $\# \vartheta^n(a)$ and the combinatorial data of $\vartheta$. We prove that there is a unique uniformity measure if and only if the Markov process $Q_\vartheta$ is ergodic in inverse time. This can be checked via standard tools in probability theory \cite{Chatterjee_Seneta}, see also \Cref{SUBSEC:Inverse-time-Markov}. We provide several sufficient conditions and give an explicit example that violates intrinsic ergodicity of both $X_\vartheta$ and $Y_\vartheta$. 
In particular, this covers and extends all the results on intrinsic ergodicity in \cite{MT-entropy}.

On a technical level, we obtain that uniformity measures have an inverse limit structure under transfer operators that represent the action of $\vartheta$, equipped with appropriate probability vectors on the inflation words. The understanding of such transfer operators is of independent interest, and we expect it to be useful for a more general study of random substitution systems.
As an intermediate step to prove that uniformity measures maximise the entropy on $Y_\vartheta$, we also show that this entropy can be obtained from the growth rate of $\# \vartheta^n(a)$ for all $a \in \mc A$, sometimes referred to as the \emph{inflation word entropy}. This unifies and generalises results from \cite{gohlke,mitchell} in a geometric setting. In fact, the equality of topological entropy and inflation word entropy holds without the assumption of recognisability.

\subsection*{Outline}

The paper is structured as follows. In \Cref{SEC:Preliminaries}, we introduce random substitutions, associated probability structures and their geometric interpretation, and we recall some background on suspension flows, induced systems, inverse-time Markov chains and conditional entropy. This provides us with all the necessary notation to properly formulate our main results in \Cref{SEC:main-results}. The equality of the topological entropy of $Y_\vartheta$ and inflation word entropy is presented in \Cref{SEC:Top-entropy}, alongside some examples that illustrate the need to change from $X_\vartheta$ to $Y_\vartheta$ for this result to hold.
We start restricting our attention to recognisable random substitutions in \Cref{SEC:structure-of-recognisable-subshifts}, where we show a structural result for the associated subshift.
\Cref{SEC:measure-transformation} is dedicated to the introduction and study of transfer operators on $X_\vartheta$ that reflect the action of $\vartheta$. This enables us to introduce the class of \emph{inverse limit measures} in \Cref{SEC:inverse-limit-measures}, generalising the class of \emph{frequency measures} studied in previous work, and to characterise their uniqueness. Interpreting uniformity measures as particular instances of limiting measures, we characterise intrinsic ergodicity in \Cref{SEC:uniformity_intrinsic_ergodicity}. In this section, we also work out a counterexample to intrinsic ergodicity in detail.

\section{Preliminaries}
\label{SEC:Preliminaries}

\subsection{Symbolic notation}

An \emph{alphabet} $\mc{A}$ is a finite collection of symbols, which we call \emph{letters}. 
We call a finite concatenation of letters a \emph{word}, and let $\mc A^{+}$ denote the set of all non-empty finite words with letters from $\mc A$. 
We write $\lvert u \rvert$ for the length of $u$ and, for each $a \in \mc A$, let $\lvert u \rvert_a$ denote the number of occurrences of $a$ in $u$. 
The \emph{Abelianisation} $\phi$ of a word $v \in \mc A^+$ is the vector $\phi(v) \in \N_0^\mc A$ with $\phi(v)_a = |v|_a$ for all $a \in \mc A$. A \emph{subword} of a word $u \in \mc A^n$ is a word $v$ such that $v = u_{[i,j]} \coloneqq u_i \cdots u_j$ for some $1 \leqslant i \leqslant j \leqslant n$. We write $|u|_v$ for the number of times that $v$ appears as a subword of $u$.

We let $\mc{A}^{\Z}$ denote the set of all bi-infinite sequences of elements in $\mc{A}$ and endow $\mc{A}^{\Z}$ with the discrete product topology. 
With this topology, the space $\mc{A}^{\Z}$ is compact and metrisable. 
We let $S$ denote the usual (left-)shift map on $\mc{A}^\Z$, given by $S(x)_n = x_{n+1}$ for all $x \in \mc A^\Z$ and $n \in \Z$. 
If $i, j \in \mathbb{Z}$ with $i \leq j$ and $x = \cdots x_{-1} x_{0} x_{1} \cdots \in \mc{A}^{\mathbb{Z}}$, then we write $x_{[i,j]} = x_i x_{i+1} \cdots x_{j}$. 
A \emph{subshift} $X$ is a closed and $S$-invariant subspace of $\mc A^{\Z}$. 
For $v \in \mc A^n$ the corresponding \emph{cylinder set} is $[v] = \{x \in X : x_{[0,n-1]} = v \}$.

For a given set $B$, we write $\# B$ for the cardinality of $B$ and let $\mc{F}(B)$ be the set of non-empty finite subsets of $B$. If $A,B \subset \mc A^+$, we write $AB = \{ uv : u\in A, v \in B\}$ for the set of all concatenations.

\subsection{Random substitutions}

There are several ways to define a random substitution. We start with a purely combinatorial definition.

\begin{definition}
    A \emph{random substitution} on a finite alphabet $\mc A$ is a set-valued function $\vartheta \colon \mc A \to \mc F(\mc A^+)$. It extends to words via
    \[
    \vartheta(v_1 \cdots v_n) = \vartheta(v_1) \cdots \vartheta(v_n),
    \]
    for $v \in \mc A^n$ and $n \in \N$, and to sets of words via $\vartheta(A) = \cup_{v \in A} \vartheta(v)$ for all $A \subset \mc F(\mc A^+)$. 
\end{definition}

As we define it here, a random substitution $\vartheta$ does not, \textit{a priori}, carry a probabilistic structure. This is because several properties of $\vartheta$ do not depend on the choices of the probabilities, and we wish to keep the flexibility to alternate between different probabilistic structures. In fact, there are several works on random substitutions that use this set-valued definition without ever assigning any probabilities \cite{gohlke, Rust_Miro_Sadun, Rust_periodic_points}.

Note that expressions like $\vartheta^2 = \vartheta \circ \vartheta$ are well defined. For convenience, we let $\vartheta^0$ denote the identity map. We call every $v \in \vartheta^n(a)$ a (level-$n$) \emph{inflation word} of type $a$.

\begin{example}
    The \emph{random Fibonacci substitution} on $\mc A= \{a,b\}$ is given by $\vartheta\colon a \mapsto \{ab,ba \}, b \mapsto \{ a\}$. We can iterate this to obtain 
    \[
    \vartheta^2(a) = \vartheta(\{ ab,ba\}) =  \vartheta(ab) \cup \vartheta(ba) = \{aba,baa \} \cup \{ aab,aba \}
    = \{ aab, aba, baa \}.
    \]
\end{example}

\begin{definition}
    Let $\vartheta$ be a random substitution on $\mc A$. The \emph{language} of $\vartheta$ is given by
    \[
    \mc L_\vartheta = \{v \in \mc A^+ : v \mbox{ is a subword of some } w \in \vartheta^n(a), \mbox{ for some } a \in \mc A, n \in \N_0 \},
    \]
    and the \emph{subshift associated with $\vartheta$} is given by
    \[
    X_{\vartheta}= \{x \in \mc A^\Z : x_{[i,j]} \in \mc L_{\vartheta} \mbox{ for all } i\leqslant j \}.
    \]
\end{definition}

Note that if $\# \vartheta(a) = 1$ for all $a \in \mc A$, our notion of a random substitution coincides with the standard definition of a substitution (identifying every singleton set with its unique element). In this case we say that $\vartheta$ is \emph{deterministic}. We recall a few basic notions about substitutions.

\begin{definition}
    Given a substitution $\theta \colon \mc A \to \mc A^+$, its \emph{substitution matrix} $M = M_\theta \in \N_0^{\mc A \times \mc A}$ is given by
    \[
    M_{ab} = |\theta(b)|_a = \phi(\theta(b))_a.
    \]
    We call $\theta$ \emph{primitive} if $M$ is a primitive matrix, that is, if $M^p$ is strictly positive for some $p \in \N$.
\end{definition}

It is sometimes convenient to regard a random substitution as a local mixture of substitutions.

\begin{definition}
    A \emph{marginal} of a random substitution $\vartheta$ is a map $\theta \colon \mc A \to \mc A^+$ such that $\theta(a) \in \vartheta(a)$ for all $a \in \mc A$. We say that $\vartheta$ is \emph{primitive} if there is some $n \in \N$ such that $\vartheta^n$ has a primitive marginal.
    We call $\vartheta$ \emph{geometrically compatible} if there is some $\lambda > 1$ and a vector $L$ with strictly positive entries, such that $L$ is a left eigenvector with eigenvalue $\lambda$ for the substitution matrix of every marginal of $\vartheta$.
\end{definition}

Primitivity is a standard assumption which ensures that the corresponding subshift is non-empty and topologically transitive \cite{rust-spindeler}, and we will assume that $\vartheta$ is primitive throughout most of this work. 

There are two special cases of geometrical compatibility that have received some attention in the past. We say that $\vartheta$ is of \emph{constant length} $\ell$ if $|v| = \ell$ for all $v \in \vartheta(a)$ and $a \in \mc A$, and we call it \emph{compatible} if each of its marginals has the same substitution matrix.
The relationship between these three conditions is illustrated in Figure \ref{FIG:1}.
We highlight that all of the inclusions here are strict, and that there is no general relation between constant length and compatible random substitutions.

\begin{remark}\label{rem:URP}
    We highlight that every geometrically compatible random substitution $\vartheta$ has the property of having \emph{unique realisation paths}: for each $u \in \mc A^n$ and $v \in \vartheta(u)$ there is a unique way to write $v = v_1 \cdots v_n$ with $v_i \in \vartheta(u_i)$; compare \cite{MT-entropy} for details. 
\end{remark}

\begin{figure}

\centering

\begin{tikzpicture} 

\node (1) at (-2.000000,0.000000) {Constant length};

\node (2) at (2.000000,0.000000) {Compatible}; 

\node (3) at (0.000000,-1.50000) {Geometrically compatibile}; 

\draw [shorten >=0.5em,shorten <=0em,-implies, double equal sign distance] (1) -- (3);

\draw [shorten >=0.5em,shorten <=0em,-implies, double equal sign distance] (2) -- (3);

\end{tikzpicture}

\caption{Implication diagram for some conditions on primitive random substitutions.} \label{FIG:1}

\end{figure}
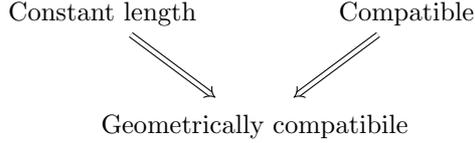

We extend the action of $\vartheta$ to bi-infinite sequences in the obvious way. More precisely, for $x \in \mc A^\Z$, let 
\[
\vartheta(x) 
= \{ \cdots v_{-2} v_{-1}. v_0 v_1 \cdots : v_i \in \vartheta(x_i) \mbox{ for all } i \in \Z \}.
\]
Here, the lower dot separates the positions indexed by $-1$ and $0$ in $\mc A^\Z$.

\begin{definition}
\label{DEF:recognisability}
    A random substitution $\vartheta$ is called \emph{recognisable} if for every $y \in X_{\vartheta}$, there exists a unique $x \in X_{\vartheta}$, a unique sequence $(v_i)_{i \in \Z}$ with $v_i \in \vartheta(x_i)$ for all $i$, and a unique $0\leqslant k < |v_0|$, such that
    \[
    S^{-k} y = \cdots v_{-2} v_{-1}. v_0 v_1 \cdots.
    \]
    We call $ (x,k,(v_i)_{i \in \Z})$ the \emph{recognisability data} of $y$ with respect to $\vartheta$.
\end{definition}

It is straightforward to verify that $\vartheta$ being recognisable implies that $\vartheta^n$ is recognisable for all $n \in \N$.
In the special case that for all $a \in \mc A$ all words in $\vartheta(a)$ have the same length, our definition of recognisability coincides with the definition used in earlier work \cite{fokkink-rust,MT-entropy}. 

Every recognisable random substitution satisfies the \emph{disjoint set condition}, meaning that for all $a \in \mc A$ and $u,v \in \vartheta(a)$, we have $\vartheta^n(u) \cap \vartheta^n(v) = \varnothing$ for all $n \in \N$. The proof of this fact carries over verbatim from the slightly more restrictive definition used in \cite[Lemma~4.5]{MT-entropy}. The disjoint set condition often simplifies the calculation of entropy, both in the topological and measure theoretic setting \cite{gohlke,MT-entropy}.

\subsection{Probabilistic aspects}

In this section, we equip a random substitution $\vartheta$ with a probabilistic structure by choosing probability vectors on each of the sets $\vartheta(a)$ with $a \in \mc A$. This approach goes back to Peyri\`{e}re \cite{peyriere} and was pursued further by Denker and Koslicki \cite{koslicki, koslicki-denker}.

\begin{definition}
    Let $\vartheta$ be a random substitution and $I = \cup_{a \in \mc A} \vartheta(a)$.
    A \emph{probability choice} for $\vartheta$ is a column stochastic matrix $\mathbf{P} \in [0,1]^{I \times \mc A}$ such that $\mathbf{P}_{u,a} = 0$ if $u \notin \vartheta(a)$. We call $\mathbf{P}$ \emph{non-degenerate} if $\mathbf{P}_{v,a} > 0$ for all $v \in \vartheta(a)$ and $a \in \mc A$.
\end{definition}

We regard $\mathbf{P}_{v,a}$ as the probability of choosing the realisation $v \in \vartheta(a)$ when applying $\vartheta$ to $a$. With some abuse of notation and in line with the usual convention, we often refer to the pair $\vartheta_{\mathbf{P}} = (\vartheta, \mathbf{P})$ as a random substitution as well. If $\vartheta(a) = \{ u_1,\ldots, u_n\}$, it is customary to represent the combined data of $\vartheta_\mathbf{P}$ as
\[
\vartheta_\mathbf{P} \colon a \mapsto
\begin{cases}
u_1 & \mbox{with probability } \mathbf{P}_{u_1,a},
\\ \hfill \vdots \hfill & \hfill \vdots \hfill
\\ u_n & \mbox{with probability } \mathbf{P}_{u_n,a}.
\end{cases}
\]

In the following we assume that $\vartheta$ has \emph{unique realisation paths}: for each $u \in \mc A^n$ and $v \in \vartheta(u)$ there is a unique way to write $v = v_1 \cdots v_n$ with $v_i \in \vartheta(u_i)$. 
This formulation is sufficient for our purposes since every geometrically compatible random substitution has this property (recall Remark~\ref{rem:URP}).

Reflecting the idea that neighbouring letters are mapped independently, we extend $\mathbf{P}$ to a countable state Markov matrix in $[0,1]^{\mc A^+ \times \mc A^+}$ via $\mathbf{P}_{v,u} = 0$ if $v \notin \vartheta(u)$, and by setting for all $u = u_1 \cdots u_n \in \mc A^n$ and $v= v_1 \cdots v_n \in \vartheta(u)$,
\[
\mathbf{P}_{v,u} = \prod_{i=1}^n \mathbf{P}_{v_i,u_i}.
\]
Recall that the splitting of $v$ into words $(v_i)_{i=1}^n$ with $v_i \in \vartheta(u_i)$ is unique due to the fact that $\vartheta$ has unique realisation paths. For a more general definition that works beyond the assumption of unique realisation paths, compare \cite{MT-entropy}.
In this notation, expressions like $\mathbf{P}^2 = \mathbf{P} \cdot \mathbf{P}$ are well defined via standard (countable state) matrix multiplication. We emphasise that such a multiplication involves only finite sums, since every column of $\mathbf{P}$ has only finitely many non-zero entries. We also note that $\mathbf{P}^n$ is a valid probability choice for $\vartheta^n$ for each $n \in \N$. 
To avoid cumbersome notation, we will write $\vartheta^n_{\mathbf{P}^n}$ for $(\vartheta^n)_{\mathbf{P}^n} = (\vartheta^n, \mathbf{P}^n)$. 

Given $u \in \mc A^+$, the Markov matrix $\mathbf{P}$ induces a stationary Markov chain $(\vartheta_{\mathbf{P}^n}^{n}(u))_{n \in \mathbb{N}}$ on some probability space $(\Omega_u, \mathcal{F}_u, \mathbb{P}_u)$ via
	\begin{align*}
	\mathbb{P}_u [\vartheta_{\mathbf{P}^{n+1}}^{n+1}(u) = w \mid \vartheta_{\mathbf{P}^n}^{n}(u) = v] = \mathbb{P}_v [\vartheta_{\mathbf{P}}(v) = w] = \mathbf{P}_{w,v},
	\end{align*} 
for all $v,w \in \mathcal{A}^{+}$ and $n \in \mathbb{N}$.  We often write $\mathbb{P}$ for $\mathbb{P}_u$ if the initial word is understood. 
In this case, we write $\mathbb{E}$ for the expectation with respect to $\mathbb{P}$.

Equipping a random substitution $\vartheta$ with a probability choice $\mathbf{P}$ also allows us to define the \textsl{substitution matrix} $M = M(\vartheta_{\mathbf{P}}) \in \mathbb{R}^{\mc A \times \mc A}$ in analogy to deterministic substitutions via
	\begin{align*}
	M_{ab}
	= \mathbb{E}[\lvert \vartheta_{\mathbf{P}}  (b) \rvert_a]
	= \sum_{v \in \vartheta(b)} \mathbf{P}_{v,b} \lvert v \rvert_a.
	\end{align*}
 If $\vartheta$ is fixed, we also write $M(\mathbf{P})$ in place of $M(\vartheta_{\mathbf{P}})$.
 A routine calculation shows that
 \[
 M(\mathbf{P} \mathbf{P}') = M(\mathbf{P}) M(\mathbf{P}').
 \]
If the matrix $M$ is primitive, Perron--Frobenius (PF) theory implies that it has a simple real (PF) eigenvalue $\lambda$ of maximal modulus and that the corresponding left and right (PF) eigenvectors $L = (L_{1}, \ldots, L_{d})$ and $R = (R_{1}, \ldots, R_{d})^{T}$ can be chosen to have strictly positive entries. We normalise the right eigenvector according to $\lVert R \rVert_{1} = 1$. If the product $L R$ is independent of $\mathbf{P}$, we usually normalise $L$ such that $L R=1$. Otherwise, we pick some arbitrary but fixed normalisation of $L$.
Like for deterministic substitutions, primitivity can be characterised purely in terms of the substitution matrix \cite[Lemma~3.2.18]{gohlke-thesis}.

\begin{lemma}
    A random substitution $\vartheta$ is primitive if and only if for some (equivalently all) non-degenerate $\mathbf{P}$ the matrix $M(\mathbf{P})$ is primitive and its PF eigenvalue satisfies $\lambda > 1$.
\end{lemma}

If $\vartheta$ is geometrically compatible, the corresponding data $\lambda$ and $L$ is precisely the PF data of $M(\mathbf{P})$ for all $\mathbf{P}$. However, the \emph{right} PF eigenvector $R$ of $M(\mathbf{P})$ does depend on $\mathbf{P}$ in the general case.

There is a special family of probability choices that is closely related to the structure of the important family of \emph{uniformity measures} which will be introduced in \Cref{DEF:uniformity-measure}.

\begin{definition}
\label{DEF:n-productivity}
    For $n \in \N_0$, the \emph{$n$-productivity distribution for $\vartheta$} is the probability choice $\mathbf{P}^{n,1}$ with
    \[
    \mathbf{P}^{n,1}_{v,a} = \frac{\#\vartheta^n(v)}{\#\vartheta^{n+1}(a)},
    \]
    for all $v \in \vartheta(a)$.
\end{definition}

The fact that $\mathbf{P}^{n,1}$ weighs inflation words according to their productivity under $\vartheta^n$ can be seen as an attempt to prepare for a uniform distribution after $n$ more applications of $\vartheta$. 

\subsection{Measures along (random) words}
Given a word $w \in \mc A^+$, let $w^\Z$ be the bi-infinite periodic sequence with period length $|w|$ and starting with the word $w$. The unique invariant measure on the orbit of $w^\Z$ with total mass $|w|$ is given by
\[
\mu_w = \sum_{i=0}^{|w| - 1} \delta_{S^i w^\Z}.
\]
It follows directly that $\mu_w([a]) = |w|_a$ for all $a \in \mc A$. Hence, writing $R^\mu := (\mu([a]))_{a \in \mc A}$ for an invariant measure $\mu$, we obtain that
\[
R^{\mu_w} = \phi(w).
\]
More generally, for any word $v \in \mc A^+$, we find that $0 \leqslant \mu_w([v]) - |w|_v \leqslant |v|$, where the maximal discrepancy $|v|$ emerges from occurrences of $v$ in $w^\Z$ that overlap several copies of $w$.

Given a random word $\omega$, the expression $\mu_\omega$ is a random measure, and we assign an invariant \emph{probability} measure, called the \emph{periodic measure representation of $\omega$} via
\[
\overline{\mu}_{\omega} = \frac{\mathbb{E}[\mu_\omega]}{\mathbb{E}[|\omega|]}.
\]

\begin{definition}
Let $(\omega_n)_{n \in \N}$ be a sequence of random words such that $\mathbb{E}[ |\omega_n| ] \to \infty$ as $n \to \infty$. We call every accumulation point of $(\overline{\mu}_{\omega_n})_{n \in \N}$ an \emph{accumulation measure} of $(\omega_n)_{n \in \N}$. If $(\overline{\mu}_{\omega_n})_{n \in \N}$ converges to some measure $\mu$, we call $\mu$ the \emph{limit measure} of $(\omega_n)_{n \in \N}$.
\end{definition}

If $\mu$ is the limit measure of $(\omega_n)_{n \in \N}$, its value on cylinder sets can be given explicitly by
\begin{equation}
\label{EQ:limit-measure-on-cylinders}
\mu([v]) = \lim_{n \to \infty} \frac{\mathbb{E}[|\omega_n|_v]}{\mathbb{E}[|\omega_n|]}.
\end{equation}
We directly obtain from \eqref{EQ:limit-measure-on-cylinders} that, whenever all realisations of $\omega_n$ are in $\mc L_{\vartheta}$ for all $n \in \N$, every corresponding accumulation measure is supported on $X_{\vartheta}$.

The sequence $(\vartheta^n_{\mathbf{P}^n}(a))_{n \in \N}$ has a well-defined limit measure $\mu_{\mathbf{P}}$, which is the same for all $a \in \mc A$. This measure is called the \emph{frequency measure} of $\vartheta_\mathbf{P}$ and is known to be ergodic under the shift map \cite{gohlke-spindeler}. 
A systematic approach to calculating the (measure theoretic) entropy of random substitution subshifts with respect to frequency measures was developed in \cite{MT-entropy}.

\subsection{Geometric hull}
Let $\vartheta$ be a primitive geometrically compatible random substitution.
The assumption of geometric compatibility gives that the Perron--Frobenius eigenvalue $\lambda$ and corresponding left eigenvector are independent of the choice of probabilities.
This allows us to choose well-defined tile lengths.
Let $\Length$ denote the tile length vector, which is some normalisation of the left PF eigenvector.
For each $w \in \mc L_{\vartheta}$, we write $\Length(w) = \sum_{a \in \mc A} \Length_a \lvert w \rvert_a$ for the \emph{geometric length} of $w$.

We will define the geometric hull of $\vartheta$ as an appropriate suspension flow. A \emph{roof function} on a subshift $X$ is a positive continuous function $\pi \colon X \rightarrow \R_{+}$ that is bounded away from $0$. The suspension flow of $X$ with roof function $\pi$ is defined by
\begin{align*}
     \operatorname{Sus}(X,\pi) = \left\{ (x,s) \,:\, x \in X,\, 0 \leq s \leq \pi (x) \right\} \subset X_{\vartheta} \times \R \tc
\end{align*}
where we identify points according to the relation $(x,s+\pi(x)) \sim (S(x),s)$, which defines an equivalence relation on the transitive closure.
For each $t \in \R$, define $T_t (x,s) = (x,s+t)$, which is well defined on $\operatorname{Sus}(X,\pi)$ via the equivalence relation above. Thus, $T = \{T_t\}$ is a one-parameter transformation group on $\operatorname{Sus}(X,\pi)$.

\begin{definition}
The \emph{geometric hull} $(Y_\vartheta,T)$ of a primitive, geometrically compatible random substitution $\vartheta$ is the suspension of $(X_{\vartheta},S)$ with roof function $\pi(x) = \Length_{x_0}$.
\end{definition}

We recall a few facts about the invariant measures on suspension systems; see \cite{Barreira_Radu_Wolf, Parry_Pollicott} for details.
Let $m$ be the Lebesgue measure on the real line.
Every $S$-invariant probability measure $\mu$ on $X_\vartheta$ can be lifted to a $T$-invariant measure $\widetilde{\mu}$ via
\[
\widetilde{\mu} = \frac{(\mu \times m)|_{Y_{\vartheta}}}{(\mu \times m)(Y_{\vartheta})}.
\]
The map $\mu \mapsto \widetilde{\mu}$ defines a bijection between the invariant measures on $(X_{\vartheta},S)$ and $(Y_{\vartheta},T)$.
Moreover, $\widetilde{\mu}$ is ergodic if and only if $\mu$ is ergodic.
By Abramov's formula, the relationship between the entropies of $\mu$ and $\widetilde{\mu}$ is given by
\[
h_{\widetilde{\mu}}(Y_\vartheta) = \frac{h_\mu(X_\vartheta)}{\mu(\pi)},
\]
using the notation $\mu(\pi):= \int \pi \dd \mu$.
For roof functions of the form $\pi(x) = \Length_{x_0}$, the normalisation factor $\mu(\pi)$ can be expressed as
\[
\mu(\pi) = \sum_{a \in \mc A} \Length_{a} \mu([a]) = \Length R^\mu.
\]

\begin{definition}
    The \emph{geometric entropy} of a shift-invariant probability measure $\mu$ on $X_{\vartheta}$ is the quantity $\hg_{\mu}:= h_{\widetilde{\mu}} = h_{\mu}/LR^\mu$. We call $\mu$ a measure of \emph{maximal geometric entropy} if $\hg_{\mu} = h_{\operatorname{top}}(Y_\vartheta)$.
\end{definition}

Note that $\mu$ is a measure of maximal geometric entropy if and only if $\widetilde{\mu}$ is a measure of maximal entropy on the suspension $(Y_{\vartheta},T)$.
Since the map $\mu \mapsto \widetilde{\mu}$ forms a bijection between the spaces of invariant measures, intrinsic ergodicity of $(Y_\vartheta,T)$ is equivalent to the existence of a unique measure of maximal \emph{geometric} entropy on $X_{\vartheta}$.
If there is a vector $R \in \R^\mc A$ such that the letter frequencies of all $x \in X_\vartheta$ are given by $R$, the normalisation factor $\mu(\pi) = \Length R$ is uniform.
In this case, the measures of maximal entropy on $Y_\vartheta$ are precisely the lifts of the measures of maximal entropy on $X_\vartheta$.
The same holds in the constant length setting, where the roof function is constant.

Another convenient interpretation of $Y_\vartheta$ is via Delone sets with finite local complexity, following the approach in \cite{Baake_Lenz_Richard}. Indeed, every $y \in Y_\vartheta$ can be represented as a (coloured) Delone set $\mc D(y) = \{\mc D_a(y) \}_{a \in \mc A}$, given by
\[
\mc D_a(y) = \{t \in \R : T_t(y) \in [a] \times \{0\} \}. 
\]
for each $a \in \mc A$. We define the intersection of a coloured set $A=\{A_a \}_{a \in \mc A}$ with a subset $B \subset \R$ as $A\cap B = \{A_a \cap B \}_{a \in \mc A}$. Similarly, $A-t = \{A_a - t \}_{a \in \mc A}$.
The image $\mc D(Y_\vartheta)$ is a space of coloured Delone sets of finite local complexity, which we equip with a (metrisable) topology in which two coloured sets are close if they agree on a large ball around the origin up to a small translation. 
It is straightforward to verify that $\mc D \colon Y_\vartheta \to \mc D(Y_\vartheta)$ is a homeomorphism and that it intertwines $T_t$ with $T'_t \colon A \mapsto A-t$ in the sense that $\mc D \circ T_t = T'_t \circ \mc D$ for all $t \in \R$.
Hence, the systems $(Y_\vartheta,T)$ and $(\mc D(Y_\vartheta), T')$ with $T'=\{T'_t \}$ are isomorphic; in particular, they have the same topological entropy.
A \emph{patch} of size $\ell \in [0,\infty)$ is an element of 
\[
\mc P_\vartheta(\ell) = \{\mc D(y) \cap [0,\ell): y \in X_\vartheta \times \{0\} \}.
\]
The \emph{patch counting function} $p_\vartheta : \R_+ \to \N$, with $p_\vartheta(\ell) = \# \mc P_\vartheta(\ell)$, satisfies
\[
p_\vartheta(\ell) = \# \{w \in \mc L_\vartheta: \Length(w_{[1,|w| -1]}) < \ell \leqslant \Length(w)  \}.
\]
Since $X_\vartheta$ (and hence $\mc D(Y_\vartheta)$) contains a point with dense orbit, we can use \cite[Thm.~1]{Baake_Lenz_Richard} to obtain the topological entropy of $Y_\vartheta$ from the exponential growth rate of $p_\vartheta$ via
\[
h_{\operatorname{top}}(Y_\vartheta)
= h_{\operatorname{top}}(\mc D(Y_\vartheta)) 
= \limsup_{\ell \to \infty} \frac{1}{\ell} \log (p_\vartheta(\ell)).
\]

Instead of considering all patterns that are close to a given length, we can restrict our attention to those that arise directly from iterating $\vartheta$. This gives rise to the following.

\begin{definition}
    Let $\vartheta$ be a geometrically compatible random substitution.
    For each $a \in \mc A$, we define the \emph{geometric inflation word entropy of type $a$} by
        \begin{equation*}
        h_a^G = \lim_{m \rightarrow \infty} \frac{1}{\Length(\vartheta^m (a))} \log (\# \vartheta^m (a))\tc
    \end{equation*}
    provided this limit exists.
\end{definition}

\subsection{Shuffle group and uniformity measures}

Let us assume that $\vartheta$ is recognisable and geometrically compatible. For $(y,s) \in Y_\vartheta$, let $( x,k,(v_i)_{i \in \Z} )$ be the recognisability data of $y$. Then, we define the recognisability data of $(y,s)$ by $(x, t, (v_i)_{i \in \Z})$, where $0\leqslant t <L(v_0) $ is the unique element such that $(y,s) = T_t (\cdots v_{-2} v_{-1}.v_0 v_1 \cdots,0)$.

The recognisable structure can be harvested to define a large number of symmetry relations that exchange inflation words of the same type and level. This idea was developed in \cite{fokkink-rust} under the additional assumption of compatibility, but the definition extends to the geometrically compatible setting. For an element $\alpha \in \operatorname{Sym}(\vartheta(a))$ of the permutation group on the set $\vartheta(a)$, the function $f_\alpha \colon Y_\vartheta \to Y_\vartheta$ is defined by replacing each word $v_i \in \vartheta(x_i)$ by $\alpha(v_i)$ whenever $x_i=a$ in the recognisability decomposition of $(y,s)$. More precisely, for $(y,s)$ with recognisability data $(x,t,(v_i)_{i \in \Z})$, we set
\[
f_{\alpha}((y,s)) = T_t( \cdots w_{-2} w_{-1}.w_0 w_1 \cdots, 0), 
\]
where $w_i = \alpha(v_i)$, whenever $ x_i=a$ and $w_i=v_i$ otherwise.
Since all elements of $\vartheta(a)$ have the same geometric length, the recognisability data of $f_{\alpha}(y,s)$ is given by $(x,t,(w_i)_{i \in \Z})$.  In particular, $f_{\alpha}$ commutes with the action of $T$. 
We call $f_{\alpha}$ a $\vartheta$-shuffle, and $\vartheta^n$-shuffles are defined accordingly.

\begin{definition}[\cite{fokkink-rust}]
For each $a \in \mc A$ and $n \in \N$, let $\Gamma_{n,a} = \{f_\alpha : \alpha \in \operatorname{Sym}(\vartheta^n(a)) \}$ and $\Gamma_n = \prod_{a \in \mc A} \Gamma_{n,a}$. We call
$\Gamma = \cup_{n \in \N} \Gamma_n$ the \emph{shuffle group} of $\vartheta$.
\end{definition}

\begin{remark}
    Note that each element in $\Gamma$ inherits the property of commuting with the action of $T$. In \cite{fokkink-rust}, the shuffle group was introduced as acting on the subshift $X_{\vartheta}$ instead of its suspension, but the naive generalisation of this definition beyond the compatible setting fails to commute with the shift action. This is why we chose to introduce the shuffle group as a symmetry of the suspension $(Y_{\vartheta},T)$ where the commutation relation with the translation is ensured by the weaker assumption of geometric compatibility.
\end{remark}

Continuity of $f \in \Gamma$ is inherited from the fact that the recognisability data of $y \in X_\vartheta$ depends continuously on $y$, see \Cref{LEM:recognisability-continuous}. Hence, $\Gamma$ is a subgroup of the automorphism group on $(Y_\vartheta,T)$.
It should be noted that shuffles are nested in the sense that $\Gamma_n$ is a subgroup of $\Gamma_{n+1}$ for all $n \in \N$. A special role will be played by those measures that respect all of these symmetry relations.

\begin{definition}
\label{DEF:uniformity-measure}
    A shift-invariant probability measure $\mu$ on $X_\vartheta$ is called a \emph{uniformity measure} if its lift $\widetilde{\mu}$ is invariant under $\Gamma$, that is, if it satisfies $\widetilde{\mu} \circ f = \widetilde{\mu}$ for all $f \in \Gamma$. 
\end{definition}

We will see later that uniformity measures always exist and have full topological support.

\subsection{The geometric substitution matrix}

A consequence of geometric compatibility is that a letter $a \in \mc A$ can be interpreted as a placeholder for an interval of length $\Length_a$. The random substitution can then be thought to act on intervals by inflating the tile by a factor $\lambda$ and randomly dissecting into intervals corresponding to letters in $\mc A$; compare \Cref{FIG:geom-compatible}. The overall length of intervals of type $a$ in $\vartheta_{\mathbf{P}}(b)$ is then given by $|\vartheta_{\mathbf{P}}(b)|_a \Length_a$, whereas the total geometric length of $\vartheta_{\mathbf{P}}(b)$ is given by $\lambda \Length_b$.
This motivates the following concept.

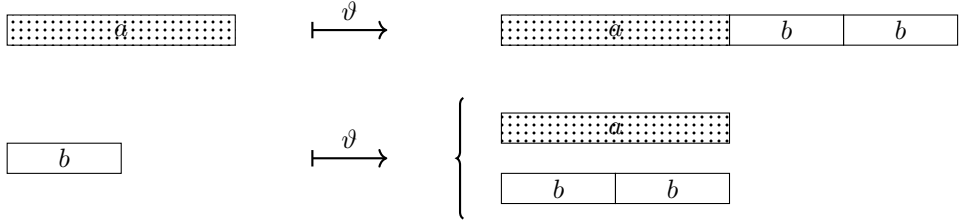
\begin{figure}
\begin{tikzpicture}
\draw[pattern = dots] (0,0) rectangle  (3,0.4);
\node[draw=white, fill=white, inner sep=2pt] at (1.5,0.2) {$a$};

\draw[|->,thick] (4.0,0.2) -- node[anchor = south]{$\vartheta$} (5,0.2);

\draw[pattern = dots] (6.5,0) rectangle (9.5,0.4);
\node[draw=white, fill=white, inner sep=2pt] at (8,0.2) {$a$};
\draw (9.5,0) rectangle node{$b$} (11,0.4);
\draw (11,0) rectangle node{$b$} (12.5,0.4);


\draw  (0, - 1.7) rectangle node{$b$} (1.5, -1.3);
\draw[|->,thick] (4.0, - 1.5) -- node[anchor = south]{$\vartheta$} (5, - 1.5);
\draw [pattern = dots] (6.5, - 1.3) rectangle (9.5, -0.9);
\node[draw=white, fill=white, inner sep=2pt] at (8,-1.1) {$a$};
\draw (6.5, - 2.1) rectangle node{$b$} (8, - 1.7);
\draw (8, - 2.1) rectangle node{$b$} (9.5, - 1.7);

\draw [
    thick,
    decoration={
        brace,
    },
    decorate
] (6,-2.3) -- (6, -0.7); 

\end{tikzpicture}
\caption{ $\vartheta \colon a \mapsto \{ abb\}, \, b \mapsto \{a,bb\}$, geometrically compatible with $\lambda = 2$ and $L = (2,1)$.}
\label{FIG:geom-compatible}
\end{figure}

\begin{definition}
The \emph{geometric substitution matrix} $Q = Q(\mathbf{P}) = Q(\vartheta_{\mathbf{P}})$ of a geometrically compatible random substitution $\vartheta_{\mathbf{P}}$ is the Markov matrix given by
\[
Q_{ab} = \frac{\Length_a}{\lambda \Length_b} M_{ab}
= \frac{\mathbb{E}[|\vartheta_{\mathbf{P}}(b)|_a] \Length_a}{\lambda \Length_b}.
\]
\end{definition}

For some applications, the Markov property poses a technical advantage over the use of the standard substitution matrix.
The geometric substitution matrix controls the expected change of geometric proportions covered by the intervals of different types. To be more precise, for a word $w$ we consider the geometric proportion vector $\phi^{\operatorname{g}}(w)$, with
\[
\phi^{\operatorname{g}}(w)_a = \frac{\Length_a \phi(w)_a}{\Length \phi(w)},
\]
and obtain via a straightforward calculation,
\[
\mathbb{E} \bigl[\phi^{\operatorname{g}}(\vartheta_{\mathbf{P}}(w)) \bigr]
= Q(\mathbf{P}) \phi^{\operatorname{g}}(w).
\]
In the same vein, it will sometimes be useful to consider a geometric analogue of the letter frequencies $R^\mu$ of an $S$-invariant measure $\mu$, given by the \emph{interval proportion vector} $\pi^\mu$, with
\begin{equation}
\label{EQ:interval-proportion-vector}
\pi^\mu_a = \frac{ \Length_a R^\mu_a}{\Length R^\mu},
\end{equation}
representing the relative geometric proportion of intervals of type $a$ witnessed by $\mu$.

\subsection{Induced transformation}

Given a compact dynamical system $(X,S)$ with invariant probability measure $\mu$ and a measurable (compact) subset $A \subset X$ with $\mu(A)>0$, the \emph{return time} $r_A\colon A \to \N \cup \{\infty\}$ is given by
\[
r_A(x) = \inf \{ n \in \N: S^n(x) \in A \}.
\]
For our purposes it is sufficient to consider the case of \emph{bounded return times}, that is, we assume that there is $r_{\max} \in \N$ such that $r_A(x) \leqslant r_{\max}$ for all $x \in A$. In this case, the \emph{induced transformation} is the dynamical system $(A,S_A,\mu_A)$, with $\mu_A(E) = \mu(A \cap E)/\mu(A)$ and 
\[
S_A(x) = S^{r_A(x)}(x),
\]
for all $x \in A$. We recall a few well-known facts about induced transformations; see for instance \cite{Sarig}. For instance, the induced measure $\mu_A$ is $S_A$-invariant, and it is ergodic if $\mu$ is an ergodic measure. Another useful tool is \emph{Kac's formula}, which states that
\[
\mu(f) : = \int f \dd \mu = \int_A \sum_{i=0}^{r_A - 1} f \circ S^i \dd \mu
\]
for all $f \in L^1(X,\mu)$. 
The corresponding statement for ergodic measures can be found in \cite[Thm.~1.7]{Sarig}. In fact, the first part of the proof in this reference shows that the statement holds for all invariant measures if $r_A$ is bounded.
Applying Kac's formula with $f\equiv 1$, we obtain that $\mu(A) = \mu_A(r_A)^{-1}$. Hence, 
\[
\int f \dd \mu
= \frac{1}{\mu_A(r_A)} \int_A \sum_{i=0}^{r_A - 1} f \circ S^i \dd \mu_A.
\]
which allows us to express $\mu$ completely in terms of $\mu_A$.

\subsection{Ergodicity of (inverse-time) Markov chains}
\label{SUBSEC:Inverse-time-Markov}

We collect a few basic properties about the convergence of inhomogenenous, finite state Markov chains in inverse time. For background and details, we refer the reader to \cite{Chatterjee_Seneta}.

\begin{definition}
    A sequence of column stochastic matrices $(P_n)_{n \in \N}$ is called \emph{ergodic} (in inverse time) if for each $n \in \N$ there exists a probability vector $\pi^n$ such that
    \[
    \lim_{k \to \infty} P_n \cdots P_{n+k} = \pi^n \mathds{1}^T.
    \]
\end{definition}

It will be convenient to measure the difference of probability vectors via the \emph{variation distance}
\[
d_V(p,q) := \frac{1}{2}|p-q|_1 = \frac{1}{2} \sum_{a \in \mc A} |p_a - q_a|,
\]
for all probability vectors $p,q$ on the state space $\mc A$.
Dobrushin's ergodic coefficient $\delta$ on a (column) stochastic matrix $Q$ is given by
\[
\delta(Q) = \max_{i,j} d_V(Q_{\cdot i},Q_{\cdot j})
= \frac{1}{2} \max_{i,j} \sum_{k} |Q_{ki}-Q_{kj}|.
\]
This coefficient satisfies several convenient properties (see \cite{Bremaud} for more details):
\begin{itemize}
\item $0 \leqslant \delta(Q) \leqslant 1$, for all Markov matrices $Q$;
\item $\delta(Q) = 0$ if and only if all columns of $Q$ coincide;
\item $\delta(Q) = 1$ if and only if there are two columns of $Q$ with disjoint support;
\item $\delta(Q_1 Q_2) \leqslant \delta(Q_1) \delta(Q_2)$
for all Markov matrices $Q_1,Q_2$ with compatible dimensions.
\end{itemize}

In fact, it is possible to express ergodicity (in inverse time) entirely in terms of this coefficient. 

\begin{theorem}[\cite{Chatterjee_Seneta}]
\label{THM:ergodicity-characterisation_and_conditions}
    The sequence $(P_n)_{n \in \N}$ is ergodic in inverse time if and only if
    \[
    \lim_{k \to \infty} \delta(P_n \cdots P_{n+k}) = 0
    \]
    for all $n \in \N$. In particular, each of the following conditions is sufficient (but not necessary) for ergodicity:
    \begin{enumerate}
    \item $\prod_{n \geqslant k} \delta(P_n) = 0$ for all $k \in \N$;
    \item $\lim_{n \to \infty} P_n = P$ for some primitive $P$.
    \end{enumerate}
\end{theorem}

    For Markov processes in forward time, the natural analogue of our definition of ergodicity is usually called ``strong ergodicity" and in fact strictly stronger than the condition that $\lim_{k \to \infty} \delta(P_n \cdots P_{n+k}) = 0$ (termed ``weak ergodicity"). In this sense, Markov processes in inverse time are more well-behaved than their analogues in forward time.

\subsection{Conditional entropy}

Let $U$ be a random variable, possibly word valued, with a  countable set $\operatorname{Im}(U)$ of possible realisations. Assume that the probability distribution of $U$ is fixed by some probability measure $\mathbb{P}$. The \emph{entropy} of $U$ with respect to $\mathbb{P}$ is given by
\[
\Ent_{\mathbb{P}}( U) =  - \sum_{u \in \operatorname{Im}(U)} \mathbb{P}[U = u] \log (\mathbb{P}[U = u]),
\]
Often, entropy is defined for a partition, but this leads to an equivalent definition if we consider partitions that are induced by countable state random variables.
We write $\Ent(U)$ for $\Ent_{\mathbb{P}}(U)$ if the probability distribution is understood. Given two random variables $U, V$, we write $\Ent(U, V) = \Ent((U,V))$ for the entropy of the random variable $(U, V)$. The \emph{conditional entropy of $U$ given $V$} with respect to $\mathbb{P}$ is defined as
\[
\Ent_{\mathbb{P}}(U | V)
= \sum_{ v \in \operatorname{Im}(V)} \mathbb{P}[ V = v] \Ent_{\mathbb{P}_{\{ V = v\}}}(U).
\]
We will freely use the following standard properties of (conditional) entropy.
\begin{enumerate}
    \item $\Ent(U) \leqslant \log(\# \operatorname{Im}(U))$, equality holds if and only if $U$ is uniformly distributed,
    \item $\Ent(U,V) =\Ent(V) +  \Ent(U|V)$, 
    \item $\Ent(U|V) \leqslant \Ent(U)$, with equality if and only if $U$ and $V$ are independent,
    \item $\Ent(U,V|W) = \Ent(V|W) +\Ent(U|V,W)$,
    \item $\Ent(U|V,W) \leqslant \Ent(U|W)$.
\end{enumerate}
We refer to \cite{Keller,Walters} for more details and background. Let us expand a bit more on how to characterise equality in the last item. By a straightforward calculation,
\[
\Ent_{\mathbb{P}}(U|V,W)
= \sum_{w \in \operatorname{Im}(W)}\mathbb{P}[W = w] \Ent_{\mathbb{P}_{\{W=w \}}}(U|V).
\]
Using the third property, we obtain that $\Ent_{\mathbb{P}}(U|V,W) = \Ent_{\mathbb{P}}(U|W)$ if and only if $U$ and $V$ are independent over $\mathbb{P}_{\{W=w \}}$ for every realisation $w$ with $\mathbb{P}[W = w]>0$.

\section{Main results}
\label{SEC:main-results}

Our first main result shows that the topological entropy of the geometric hull can be obtained by counting inflation words. We emphasise that this does not require $\vartheta$ to be recognisable.
This generalises and unifies the results on topological entropy in \cite{gohlke,mitchell}.

\begin{itheorem}\label{THM:geom-inf-entropy}
    Let $\vartheta$ be a primitive and geometrically compatible random substitution. Then, for all $a \in \mc A$, the geometric inflation word entropy $h_a^G$ exists and coincides with $h_{\operatorname{top}} (Y_{\vartheta})$.
\end{itheorem}

In the symbolic setting, it was shown in \cite{MT-entropy,mitchell-thesis} that for all primitive random substitutions that are compatible or constant length, there exists a sequence of frequency measures that converges weakly to a measure of maximal entropy.
As a consequence of \Cref{THM:geom-inf-entropy}, we will obtain the analogous result in the geometrically compatible setting, on the geometric hull.

\begin{icorollary}
\label{COR:frequency-measure-approximation}
    Let $\vartheta$ be a geometrically compatible random substitution with associated geometric hull $Y_{\vartheta}$.
    Then, there exists a sequence $(\mu_m)_m$ of frequency measures (on the symbolic hull $X_{\vartheta}$) whose push-forwards converge weakly to a measure of maximal entropy on $Y_{\vartheta}$.
\end{icorollary}

In general however, the class of frequency measures is too small to contain the measure of maximal (geometric) entropy; an example with a unique MME that is not a frequency measure is presented in \Cref{Exa:IE-MME-not-freq}.
A more adequate family is given by the \emph{inverse limit measures}, presented in \Cref{SEC:inverse-limit-measures}. In particular, this class contains all uniformity measures.

\begin{itheorem}
\label{THM:main_uniformity_maximal}
     Let $\vartheta$ be a primitive, geometrically compatible and recognisable random substitution. Then, the measures of maximal geometric entropy on $(X_{\vartheta},S)$ are precisely the uniformity measures.
\end{itheorem}

Since uniformity measures have full topological support, we conclude that $X_\vartheta$ is (geometric) entropy--minimal, that is, all proper invariant subshifts have a smaller (geometric) entropy.

\begin{itheorem}
    \label{THM:main_Markov_chain_condition}
    Let $\vartheta$ be a primitive, geometrically compatible and recognisable random substitution and let $Q_n = Q(\mathbf{P}^{n,1})$, where $\mathbf{P}^{n,1}$ is the $n$-productivity distribution for $\vartheta$ for all $n \in \N_0$.
    Then, there is a unique uniformity measure if and only if the Markov chain $(Q_n)_{n \in \N_0}$ is ergodic in inverse time.
\end{itheorem}

If $\vartheta$ is compatible or of constant length, then the measures of maximal geometric entropy are precisely the measures of maximal entropy.
We note that several isolated examples were shown to be intrinsically ergodic in \cite{MT-entropy}.
In all of these cases, the $n$-productivity distributions are uniform distributions and the Markov chain is trivially ergodic.
In fact, this is true whenever $\vartheta$ is compatible.

\begin{icorollary}
\label{COR:main_rpc}
If $\vartheta$ is primitive, compatible and recognisable, both $(X_{\vartheta},S)$ and $(Y_\vartheta,T)$ are intrinsically ergodic. The unique measure of maximal entropy is the frequency measure $\mu_\mathbf{P}$ where $\mathbf{P}$ is the uniform distribution on $\vartheta(a)$ for all $a \in \mc A$.
\end{icorollary}

In general, this is not true if compatible is relaxed to geometrically compatible. Even in the constant length setting, we can find examples where $(Q_n)_{n \in \N_0}$ is not ergodic, and therefore obtain cases where both $(X_{\vartheta},S)$ and $(Y_\vartheta,T)$ are not intrinsically ergodic. A specific example for which this occurs is worked out at the end of \Cref{SEC:uniformity_intrinsic_ergodicity}.

\section{Topological entropy of the geometric hull}
\label{SEC:Top-entropy}

\subsection{Geometric inflation word entropy}

The proof of \Cref{THM:geom-inf-entropy} follows a similar line of arguments to those in \cite[Thm.~4.1]{mitchell}, adapted to the geometric setting.

\begin{prop}
\label{PROP:h-top-upper-bound}
    Let $\vartheta$ be a geometrically compatible random substitution with associated geometric hull $Y_{\vartheta}$.
    Let $k, m \in \N$ and set
    \begin{align*}
        h^{m,k}_{\max} \coloneqq \max_{a \in \mc A} \max_{u \in \vartheta^k (a)} \frac{\log (\# \vartheta^m (u))}{\Length (\vartheta^m (u))}\tp
    \end{align*}
    Then, the following inequality holds:
    \begin{align*}
        h_{\operatorname{top}}(Y_\vartheta) \leq
        \frac{\lambda^m}{\lambda^m - 1} h^{m,k}_{\max}\tp
    \end{align*}
\end{prop}

\begin{proof}
    Fix $k,m \in \N$ and let $n \in \N$.
    For every legal word $w \in \mc L_{\vartheta}$, we have that $\Length (\vartheta^m (w)) = \lambda^m \Length (w)$.
    For each $a \in \mc A$ and $w \in \vartheta^k(a)$, we let $h_{w}^{m}$ be the number such that
    $\# \vartheta^m (w) = \exp ( h_{w}^{m} \lambda^m \Length (w))$. 
    By definition, we have that $h^{m,k}_{\max} = \max_{a \in \mc A} \max_{w \in \mc \vartheta^k (a)} h_{w}^{m}$.
    Note that if $v = v_1 \cdots v_r$ is the concatenation of level-$k$ inflation words (that is, $v_i \in \vartheta^k (a_i)$ for some $a_i \in \mc A$, for all $i \in \{1,\ldots,r\}$), then
    \begin{align}\label{EQ:level-k-entropy-contribution}
        \# \vartheta^m (v) 
        = \prod_{i=1}^{r} \# \vartheta^m (v_i) 
        = \prod_{i=1}^{r} \exp \left(h_{v_i}^{m} \lambda^m \Length (v_i)\right) 
        = \exp \left(\sum_{i=1}^{r} h_{v_i}^{m} \lambda^m \Length (v_i)\right)
        \leq \exp \left(h^{m,k}_{\max} \lambda^m \Length (v)\right)
        \tp
    \end{align}
    By definition of $Y_{\vartheta}$, every patch of length $\lambda^m n$ is contained in the image of a patch with length $n + L_{\max}$, where $L_{\max} = \max_{a \in \mc A} L_a$.
    Moreover, the image of such a patch contains at most $C_1$ patches of length $\lambda^m n$, where $C_1$ is a constant dependent on $m$ but not $n$. This is because patches have a control point at the origin by definition, and these control points occur with bounded distances.
    Hence, the number of patches of length $\lambda^m n$ is bounded above by
    \begin{align}\label{EQ:patch-counting-UB}
        \# \mc P_{\vartheta} (\lambda^m n) \leq C_1 \sum_{\substack{v \in \mc L_{\vartheta}\\L(v_{[1,|v|-1]}) < n+L_{\max} \leqslant L(v)}} \# \vartheta^m (v) \tp
    \end{align}
    For every $v \in \mc L_{\vartheta}$, there exists a $w \in \mc L_{\vartheta}$ that is the concatenation of level-$k$ inflation words such that $v$ is contained in $w$.
    Moreover, such a $w$ can be chosen with length at most $\lvert v \rvert + 2 \lvert \vartheta^k \rvert$, where $\lvert \vartheta^k \rvert = \max_{a \in \mc A} \max_{s \in \vartheta^k (a)} \lvert s \rvert$.
    Therefore, the geometric length $L (w)$ is at most $n + (2 \lvert \vartheta^k \rvert +1) L_{\text{max}}$.
    Thus, it follows by \eqref{EQ:level-k-entropy-contribution} that there is a constant $C_2$ such that
    \begin{align*}
        \# \vartheta^m (v) \leq \# \vartheta^m (w) \leq \exp \left( h^{m,k}_{\max} \lambda^m \Length (w) \right) \leq \exp \left( h^{m,k}_{\max} \lambda^m (n + C_2) \right)
    \end{align*}
    Substituting this expression into \eqref{EQ:patch-counting-UB} gives
    \begin{align*}
        \# \mc P_{\vartheta} (\lambda^m n) \leq C_1 (\# \mc P_{\vartheta} (n+L_{\max})) \exp \left( h^{m,k}_{\max} \lambda^m (n + C_2) \right)
    \end{align*}
    and so it follows that
    \begin{align*}
        \limsup_{n \rightarrow \infty} \frac{1}{\lambda^m n} \log (\# \mc P_{\vartheta} (\lambda^m n)) \leq \limsup_{n \rightarrow \infty} \frac{1}{\lambda^m n} \log (\# \mc P_{\vartheta} (n+L_{\max})) + \limsup_{n \rightarrow \infty} \left( 1 + \frac{C_2}{n} \right) h^{m,k}_{\max}\tp
    \end{align*}
    Hence, we obtain
    \begin{align*}
        h_{\operatorname{top}} (Y_{\vartheta}) \leq \frac{1}{\lambda^m} h_{\operatorname{top}} (Y_{\vartheta}) + h^{m,k}_{\max}\tc
    \end{align*}
    and rearranging then gives the desired result.
\end{proof}

\begin{proof}[Proof of \Cref{THM:geom-inf-entropy}]
For $n,k \in \N$ and $a \in \mc A$, let 
\[
h^{n,k}_a = \max_{u \in \vartheta^k(a)} \frac{\log (\# \vartheta^n(u))}{\Length(\vartheta^n(u))}
\]
and let $u^{n,k}_a \in \vartheta^k(a)$ be a word for which this maximum is achieved. We let $h^{n,k}_{\max} = \max_{a \in \mc A} h^{n,k}_a$ and $h^{n,k}_{\min} = \min_{a \in \mc A} h^{n,k}_a$.
Since we assumed $\vartheta$ to be primitive, there is a number $N \in \N$ such that $\vartheta^N$ has a marginal with strictly positive substitution matrix. Given $a \in \mc A$, we can hence choose a word $w = w_1 \cdots w_m \in \vartheta^N(a)$ that contains every letter in $\mc A$.
Assuming $k > N$, we note that $u^{n,k-N}_{w_j} \in \vartheta^{k-N}(w_j)$ and hence we can pick a realisation $v^k \in \vartheta^{k-N}(w) \subseteq \vartheta^k(a)$ of the form
\begin{align}\label{eq:v-word-def}
    v^k = u_{w_1}^{n,k-N} \cdots u_{w_m}^{n,k-N}.
\end{align}
Since $\vartheta$ is geometrically compatible, it has unique realisation paths (recall \Cref{rem:URP}), so we have that
\begin{equation}
\label{EQ:inflation-word-decomposition}
\log (\# \vartheta^n(v^k)) = \sum_{j=1}^{m} \log (\# \vartheta^n(u^{n,k-N}_{w_j}))
= \sum_{j=1}^m h^{n,k-N}_{w_j} \Length(\vartheta^n(u^{n,k-N}_{w_j})) .
\end{equation}
Observe that $\Length(\vartheta^n(u^{n,k-N}_{w_j})) = \lambda^{n+k-N} \Length(w_j)$ and that
$
\Length(\vartheta^n(v^k)) = \sum_{j=1}^{m} \Length(\vartheta^n(u^{n,k-N}_{w_j})).
$
Hence, for each $j$ (and recalling that the length $m = |w|$ depends only on the letter $a$), we find that
\[
\frac{ \Length(\vartheta^n(u^{n,k-N}_{w_j}))}{\Length(\vartheta^n(v^k))}
\geqslant \frac{\min_{b \in \mc A} \Length_b}{m \max_{b \in \mc A} \Length_b}=:q_a.
\]
With $q = \min_{b \in \mc A} q_b >0$, the fact that $w$ contains every letter implies that $h^{n,k-N}_{\max}$ appears at least with weight $q \Length(\vartheta^n(v))$ in the last sum in \eqref{EQ:inflation-word-decomposition}. Using that $v^k \in \vartheta^k(a)$, we thus obtain
\begin{equation}
\label{EQ:recursion-for-max-min-entropy-approximants}
h^{n,k}_a \geqslant \frac{\log (\# \vartheta^n(v^k))}{\Length(\vartheta^n(v^k))} \geqslant q h^{n,k-N}_{\max} + (1-q) h^{n,k-N}_{\min},
\end{equation}
for all $a \in \mc A$. In particular, the same lower bound holds for $h^{n,k}_{\min}$ in place of $h^{n,k}_{a}$. Recall that, by \Cref{PROP:h-top-upper-bound}, we have
\[
h^{n,j}_{\max} \geqslant \frac{\lambda^n-1}{\lambda^n} h_{\operatorname{top}}(Y_{\vartheta}) =: h^{(n)}
\]
for all $j \in \N$. Together with \eqref{EQ:recursion-for-max-min-entropy-approximants}, we therefore find
\[
h^{n,k}_{\min} \geqslant q h^{(n)} + (1-q) h^{n,k-N}_{\min}.
\]
By applying this inequality repeatedly and comparing with \eqref{EQ:recursion-for-max-min-entropy-approximants}, we obtain that
\begin{align*}
    \frac{\log(\# \vartheta^n (v^k))}{L (\vartheta^n (v^k))} \geq q \sum_{j=0}^{\ell(k)-1} (1-q)^j h^{(n)} + h_{\min}^{n,r(k)}\tc
\end{align*}
where $\ell(k) = \lfloor k/N\rfloor$ and $r(k)=k-N\ell(k)$.
Letting $k \rightarrow \infty$ in the above, and bounding the second term below by zero, we thus obtain that
\[
\liminf_{m \to \infty} \frac{\log (\# \vartheta^m(a))}{\Length(\vartheta^m(a))}
= \liminf_{k \to \infty} \frac{\log (\# \vartheta^{n+k}(a))}{\Length(\vartheta^{n+k}(a))}
\geqslant \liminf_{k \to \infty} \frac{\log (\# \vartheta^n(v^k))}{\Length(\vartheta^n(v^k))} \geqslant h^{(n)} \xrightarrow{n \to \infty} h_{\operatorname{top}}(Y_{\vartheta}).
\]
This shows the lower bound for the inflation word entropy. The upper bound
\[
\limsup_{m \to \infty} \frac{\log (\# \vartheta^m(a))}{\Length(\vartheta^m(a))} 
\leqslant h_{\operatorname{top}}(Y_{\vartheta}),
\]
is immediate because all words in $\# \vartheta^m(a)$ are legal patterns of size $\Length(\vartheta^m(a))$ in $Y_{\vartheta}$.
\end{proof}

\subsection{Measure theoretic entropy of frequency measures}

In this section, we prove that frequency measures can be used to approximate the topological entropy to an arbitrary level (\Cref{COR:frequency-measure-approximation}).
This follows by combining \Cref{THM:geom-inf-entropy} with the results on entropy of frequency measures proved in \cite{MT-entropy}.

\begin{definition}
For a primitive random substitution $\vartheta_{\mathbf{P}}$ on a finite alphabet $\mathcal A$ and $m \in \N$, we let $H_{\mathbf{P}^m} = (H_{\mathbf{P}^m,a})_{a \in \mathcal A}$ denote the row vector with entries $H_{\mathbf{P}^m,a} = \Ent(\vartheta^m_{\mathbf{P}^m}(a))$ for all $a \in \mathcal A$.
\end{definition}

\begin{theorem}[{\cite[Thm.~3.5]{MT-entropy}}]\label{THM:MT-entropy}
    Let $\vartheta_{\bm P}$ be a primitive and geometrically compatible random substitution, with Perron--Frobenius eigenvalue $\lambda$ and right eigenvector $R$.
    Then, for all $m \in \N$,
    \begin{align*}
        \frac{1}{\lambda^m} H_{\mathbf{P}^m} R \leq h_{\mu_{\bm P}} (X_{\vartheta}) \leq \frac{1}{\lambda^m-1} H_{\mathbf{P}^m} R .
    \end{align*}
\end{theorem}

\begin{proof}[Proof of \Cref{COR:frequency-measure-approximation}]
    Let $(\mu_m)_m$ be the sequence of frequency measures corresponding to equi-distributed probabilities on $\vartheta^m$.
    For each $m \in \N$, let $R_m$ denote the right Perron--Frobenius eigenvector of the substitution matrix for the choice of probabilities associated with the measure $\mu_m$.
    Further, we let $\widetilde{\mu}_m$ denote the push-forward of $\mu_m$ onto $Y_{\vartheta}$.

    Let $\varepsilon > 0$. By \Cref{THM:geom-inf-entropy}, there is an $M \in \N$ such that for all $m \geq M$ and all $a \in \mc A$, we have $\log (\# \vartheta^m (a))/\lambda^m > \Length_a (h_{\operatorname{top}} (Y_{\vartheta})-\varepsilon)$, noting that $\Length (\vartheta^m (a)) = \lambda^m\Length_a$.
    Thus, it follows by \Cref{THM:MT-entropy} and Abramov's formula that for all $m \geq M$, we have
    \begin{align*}
        h_{\widetilde{\mu}_m} (Y_{\vartheta}) 
        \geq \frac{1}{\Length R_m} \sum_{a \in \mc A} R_{m,a} \frac{\log (\# \vartheta^m (a))}{\lambda^m} 
        > \frac{1}{\Length R_m} \sum_{a \in \mc A} R_{m,a} \Length_a (h_{\operatorname{top}} (Y_{\vartheta})-\varepsilon)
        = h_{\operatorname{top}} (Y_{\vartheta}) - \varepsilon\tc
    \end{align*}
    Since this holds for all $\varepsilon>0$, we conclude that $h_{\widetilde{\mu}_m} (Y_{\vartheta}) \rightarrow h_{\operatorname{top}} (Y_{\vartheta})$ as $m \rightarrow \infty$.
    Thus, by the compactness of the space of shift-invariant measures on $X_{\vartheta}$, we conclude that there is a sequence of frequency measures whose push-forwards converge weakly to a measure of maximal entropy on $Y_{\vartheta}$.
    \end{proof}

\subsection{Relationship between entropy of symbolic and geometric hulls}

\begin{definition}
Let $X$ be a subshift over a finite alphabet $\mc A$. We say that a probability vector $\eta \in [0,1]^{\mc A}$ is a \emph{letter frequency vector} for $X$ if there exists $x \in X$ such that $\eta = \eta(x)$, where
\[
\eta(x)_a = \lim_{n \to \infty} \frac{|x_{[-n,n]}|_a}{2n+1},
\]
for all $a \in \mc A$.
\end{definition}
\begin{lemma}
\label{LEM:symb-geom-entropy-relation}
Let $Y$ be the suspension of a subshift $X$, with associated length vector $\Length \colon \mc A \to \R_+$ and roof function $\pi(x) = \Length_{x_0}$.
    Further, let $\eta^{-}$ and $\eta^{+}$ denote letter frequency vectors that minimise and maximise the quantities $\Length \eta^-$ and $\Length \eta^+$, respectively.
    Then, the following inequalities hold:
    \begin{equation*}
        \frac{1}{\Length \eta^{+}} h_{\operatorname{top}} (X) \leq h_{\operatorname{top}} (Y) \leq \frac{1}{\Length \eta^{-}} h_{\operatorname{top}} (X)\tp
    \end{equation*}
\end{lemma}
\begin{proof}
    By the variational principle for suspension flows and Abramov's formula, we have
    \[
    h_{\operatorname{top}}(Y) = \sup h_{\widetilde{\mu}} = \sup 
    \frac{h_\mu}{\Length R^{\mu}},
    \]
    where the supremum is taken over all ergodic measures. By ergodicity, each such measure satisfies that $R^\mu = \eta(x)$ for some $x \in X$ and therefore is a letter frequency vector for $X$. Hence, $\Length \eta^- \leqslant \Length R^\mu \leqslant \Length \eta^+$, implying the desired bounds.
\end{proof}

\begin{corollary}\label{COR:entropy-equiv}
    If $\vartheta$ is a constant length or compatible random substitution and $\Length$ is the left eigenvector, normalised such that $LR = 1$, then $h_{\operatorname{top}} (X_{\vartheta}) = h_{\operatorname{top}} (Y_{\vartheta})$.
\end{corollary}

In the general geometrically compatible setting, it is possible for the inequalities in \Cref{LEM:symb-geom-entropy-relation} to be strict: we give an explicit example in the next section.
In fact, there exist examples of geometrically compatible random substitutions (which are not compatible or constant length) for which the measure of maximal entropy is not a measure of maximal geometric entropy: see \Cref{exa:golden-mean-shift}.

\subsection{Examples}

The following examples illustrate how \Cref{THM:geom-inf-entropy} can be used to obtain the topological entropy for subshifts of random substitutions that are neither compatible nor constant length.

\begin{example}
    Let $\vartheta$ be the primitive random substitution defined by 
    \[
    \vartheta \colon a \mapsto \{ab,ac\},\, b,c \mapsto \{a\},
    \]
    which is geometrically compatible but neither constant length nor compatible.
    We show that
    \begin{align*}
        h_{\operatorname{top}} (X_{\vartheta}) = h_{\operatorname{top}} (Y_{\vartheta}) = \frac{1}{\tau^2} \log 2\tc
    \end{align*}
    where $\tau = (1+\sqrt{5})/2$ denotes the golden ratio.
    To this end, we first observe that the vector
    \begin{align*}
        L = \frac{1}{1+\tau^{-2}} \left(\tau,1,1\right)
    \end{align*}
    is the left Perron--Frobenius eigenvector of the substitution matrix for an appropriate choice of normalisation.
    Since the frequency of $a$ in every element of $X_{\vartheta}$ is equal to $\tau^{-1}$, every letter frequency vector $\eta$ satisfies $L \eta = 1$.
    Thus, it follows by \Cref{LEM:symb-geom-entropy-relation} that $h_{\operatorname{top}} (X_{\vartheta}) = h_{\operatorname{top}} (Y_{\vartheta})$.
    By \Cref{THM:geom-inf-entropy}, both coincide with the quantity $h_a^{G}$.
    Now, note that for all $n \in \N$, we have $\vartheta^{n+1} (a) = \vartheta^n (ab) \cup \vartheta^n (ac) = \vartheta^n (ab)$, so
    \begin{align*}
        \# \vartheta^{n+1} (a) = \# \vartheta^n (ab) = \# \vartheta^n (a) \# \vartheta^n (b) = \# \vartheta^n (a) \# \vartheta^{n-1} (a)\tp
    \end{align*}
    Iterating this identity, and noting that $\# \vartheta^{1} (a) = 2$, we obtain that $\# \vartheta^{n} (a) = 2^{F_n}$. 
    Hence,
    \begin{align*}
        h_a^G = \lim_{n \rightarrow \infty} \frac{1}{L(\vartheta^n (a))} \log (\# \vartheta^n (a)) = \lim_{n \rightarrow \infty} \frac{F_n}{\tau^{n+1} (1+\tau^{-2})^{-1}} \log 2 = \lim_{n \rightarrow \infty} \frac{\tau^2+1}{\tau^3 \sqrt{5}} \log 2= \frac{1}{\tau^2} \log 2\tc
    \end{align*}
    where in the second last equality we have applied Binet's formula and in the final equality we have used the identity $(\tau^2+1) = \tau \sqrt{5}$.
    This establises the desired result.
\end{example}

The following example demonstrates that, in general, the topological entropy of the symbolic hull need not coincide with the geometric inflation word entropy if the symbolic length $\lvert \vartheta (a) \rvert$ is not well defined for all $a \in \mc A$.

\begin{example}\label{exa:golden-mean-shift}
Let $\vartheta$ be the random substitution from \Cref{FIG:geom-compatible}, defined by
\[
\vartheta \colon a \mapsto \{abb \},
\quad b \mapsto \{ a, bb \}.
\]
The random substitution $\vartheta$ is geometrically compatible, but neither compatible nor constant length.
Further, in contrast to the previous example, the symbolic inflation word lengths $\lvert \vartheta (a_i) \rvert$ are not well defined for all letters $a_i$.
All marginals of $\vartheta$ have a substitution matrix with Perron--Frobenius eigenvalue $\lambda = 2$ and left eigenvector $L = (2,1)$. We take this vector $L$ to define $Y_\vartheta$.

By \Cref{THM:geom-inf-entropy}, we can calculate $h_{\operatorname{top}} (Y_\vartheta)$ via the geometric inflation word entropy.
Note that $\vartheta (a) \subseteq \vartheta (bb)$ and so it follows inductively that $\# \vartheta^m (b) = (\# \vartheta^{m-1} (b))^2 = (\# \vartheta (b))^{2^{m-1}} = 2^{2^{m-1}}$ for all $m \in \N$.
Hence,
\begin{align*}
    \frac{1}{L (\vartheta^m (b))} \log (\# \vartheta^m (b)) = \frac{1}{2^m} \log \big(2^{2^{m-1}} \big) = \frac{1}{2} \log 2
\end{align*}
for all $m \in \N$, so we conclude that $h_{\operatorname{top}} (Y_{\vartheta}) = \log (2)/2$.

We now demonstrate that the (unique) measure of maximal entropy on the symbolic hull does not correspond to a measure of maximal entropy on the associated geometric hull.
Let $X'$ be the shift of finite type defined over the alphabet $\{a,b_0,b_1\}$ by the forbidden word set $\mc F = \{a b_1, b_1 b_1, b_0 b_0, b_0 a\}$ and let $\pi \colon \{a,b_0,b_1\}^{\Z} \rightarrow \{a,b\}$ be the factor map given by $\pi(x)_i = \pi'(x_i)$, where $\pi' (a) = a$, $\pi' (b_0) = \pi' (b_1) = b$.
It can easily be verified that $\pi (X') = X_{\vartheta}$.
Further, $\pi$ is one-to-one everywhere except on the sequence $b^{\Z}$.
Thus, the subshift $X_{\vartheta}$ is intrinsically ergodic with the unique measure of maximal entropy $\mu$ given by the pushforward by $\pi$ of the Parry measure on $X'$.
Hence, by standard results on the entropy of shifts of finite type, we obtain that $h_{\mu} (X_{\vartheta}) = h_{\operatorname{top}} (X_{\vartheta}) = \log \tau$, where $\tau$ is the golden ratio.
Moreover, the letter frequency vector $R^{\mu}$ associated with $\mu$ is given by
\begin{align*}
    R^{\mu} = \left( \frac{\tau}{\tau+2}, \frac{2}{\tau+2} \right)
\end{align*}
and so we have
\begin{align*}
    L R^{\mu} = \frac{2\tau+2}{\tau+2} \tp
\end{align*}
Thus, it follows by Abramov's formula that the lift $\widetilde{\mu}$ of $\mu$ onto $Y_{\vartheta}$ has entropy
\begin{align*}
    h_{\widetilde{\mu}} (Y_{\vartheta}) = \frac{h_{\mu} (X_{\vartheta})}{L R^{\mu}} = \left( \frac{\tau+2}{2\tau+2} \right) \log \tau\tp
\end{align*}
Since
\begin{align*}
    0.3325 \approx \left( \frac{\tau+2}{2\tau+2} \right) \log \tau < \frac{1}{2} \log 2 \approx 0.3466\tc
\end{align*}
it follows that $h_{\widetilde{\mu}} (Y_{\vartheta}) < h_{\operatorname{top}} (Y_{\vartheta})$, and so $\widetilde{\mu}$ is not a measure of maximal entropy for $Y_{\vartheta}$.
\end{example}

\section{Structure of recognisable subshifts}
\label{SEC:structure-of-recognisable-subshifts}

From now on, we will assume that $\vartheta$ is recognisable.
For the special case of compatible random substitutions, it was noted in \cite{fokkink-rust} that there is an equivalent local formulation of recognisability. This is in line with the following result.

\begin{lemma}
\label{LEM:recognisability-continuous}
In the definition of recognisability, the data $x,k$ and $(v_i)_{i \in \Z}$ depend continuously on $y$.
\end{lemma}

\begin{proof}
Assume that $y^n \in X_{\vartheta}$ satisfies $\lim_{n \to \infty} y^n = y$. For $n \in \N$, let $D^n = (x^n, k^n, (v^n_i)_{i \in \Z})$ be the recognisability data of $y^n$. Due to compactness, it suffices to show that the recognisability data of $y$ is the only accumulation point of $(D^n)_{n \in \N}$. Up to restricting to a subsequence, we can in fact assume that $D^n$ converges to some $(x,k,(v_i)_{i \in \Z})$. In particular, $(v_i^n)_{n \in \N}$ and $(x^n_i)_{n \in \N}$ are eventually constant for all $i\in \Z$, implying that $v_i \in \vartheta(x_i)$.
For convenience, let $v^n = \cdots v^{n}_{-1}. v^n_0 \cdots$ and similarly for $v$. Given a ball $B$ around the origin, we can choose $n$ large enough that $y|_B = y_n|_B = S^k v^n|_B = S^k v|_B $. Since $B$ was arbitrary, we have $y = S^k v$, and hence $(x,k,(v_n)_{n \in \Z})$ is indeed the unique recognisability data of $y$.
\end{proof}

\begin{remark}
As an immediate consequence, we obtain that if $\vartheta$ is recognisable and of constant length $\ell$, the $\ell$-adic odometer $ \Omega_\ell \subset \prod_{n=1}^\infty \Z/\ell^n\Z$, with addition $+1$, is a topological factor of $(X_{\vartheta},S)$. Indeed, we obtain an explicit factor map $x \mapsto (k_n)_{n \in \N}$, where $k_n$ is the unique number in $[0,\ell^n)$ such that $x \in S^{k_n}(\vartheta^n(X_{\vartheta}))$. In particular, $(X_{\vartheta},S)$ cannot be topologically mixing. For related results on topological mixing under the assumption of compatibility, we refer the reader to \cite{Rust_Miro_Sadun}. In particular, the observation in this remark slightly extends \cite[Cor.~14]{Rust_Miro_Sadun} by dropping the assumption of compatibility.
\end{remark}

For recognisable random substitutions, it is often convenient to consider the compact subset $\vartheta(X_{\vartheta})$ and the associated induced transformation. This gives rise to the following structure.

\begin{lemma}
\label{LEM:conjugation}
Let $\vartheta$ be a recognisable random substitution. Assume that $A = \vartheta(X_{\vartheta})$ and let $(A,S_A)$ be the induced transformation of $(X_{\vartheta}, S)$ on $A$. Then $(A,S_A)$ is topologically conjugate to the shift map $S$ on the space $Z_{\vartheta} \subset \mc B^{\Z}$, where $\mc B =\{(a,v) : a \in \mc A, v \in \vartheta(a) \}$ and
\[
Z_{\vartheta} = \{ (x_i,v_i)_{i \in \Z} \in \mc B^{\Z}: (x_i)_{i \in \Z} \in X_{\vartheta} \}.
\]
\end{lemma}

\begin{proof}
    We give the conjugacy map $\psi\colon \vartheta(X_{\vartheta}) \to Z_\vartheta$ explicitly. Given $y \in \vartheta(X_{\vartheta})$, recognisability implies that there is a unique element $x \in X_{\vartheta}$ and a unique sequence $(v_i)_{i \in \Z}$ with $v_i \in \vartheta(x_i)$ such that $y = (v_i)_{i \in \Z}$ is a concatenation of these inflation words. 
    Set
    \[
    \psi(y) = (x_i,v_i)_{i \in \Z}.
    \]
    Both injectivity and surjectivity are readily verified. The inverse is $\psi^{-1}((x_i,v_i)_{i \in \Z}) = (v_i)_{i \in \Z}$. 
    Continuity of $\psi$ follows directly from \Cref{LEM:recognisability-continuous}.
    In the above notation, the first return time $r_A$ of $y$ satisfies $r_A(y) = |v_0|$ and therefore, $S_A y = (v_{i+1})_{i \in \Z}$. This implies
    \[
    \psi(S_A y) = (x_{i+1},v_{i+1})_{i \in \Z} = S \psi(y),
    \]
    which is the required conjugacy relation.
\end{proof}

Thanks to \Cref{LEM:conjugation}, we often identify $(A,S_A)$ with $(Z_\vartheta,S)$, with slight abuse of notation. Due to the established topological conjugacy, this has no effect on results concerning entropy and ergodic measures, other than a renaming of objects.
Note that the same construction works for $A_n = \vartheta^n(X_{\vartheta})$ instead of $A$.

The shuffle group has a convenient representation in terms of the family of subshifts ($A_n,S_{A_n}$), with $n \in \N$. 
Recall that if $f \in \Gamma_n$, then for each $a \in \mc A$, there is a permutation $\alpha_a (\vartheta^n (a))$ that replaces each word $v \in \vartheta^n (a)$ in the inflation word decomposition of $x \in X_{\vartheta}$ with $\alpha_a (v)$.
Thus, $f$ acts on $(x_i,v_i)_{i \in \Z}$ as follows:
\[
    f(y) = (x_i, w_i)_{i \in \Z}, \quad w_i = \alpha_{x_i} (v_i).
\]
This is consistent with the earlier definition, in the sense that for such $y$,
\[
f((y,t)) = (f(y),t),
\]
for all $t \in \R$. It follows that $\widetilde{\mu}$ is invariant under $f$ precisely if the same holds for $\mu_{A_n}$. \begin{corollary}
\label{COR:Uniformity_char}
$\mu$ is a uniformity measure if and only if $\mu_{A_n}$ is invariant under $\Gamma_n$ for all $n \in \N$.
\end{corollary}

\section{Measure transformation}
\label{SEC:measure-transformation}

The assumption of recognisability is consistent with the idea of a desubstitution. We implement this on the level of measures in the following way. Let $\mc M$ denote the space of $S$-invariant Borel probability measures on $X_\vartheta$. For $\nu \in \mc M$, let $\nu_A$ be the corresponding induced measure on $A=\vartheta(X_{\vartheta})$. Using the identification in Lemma~\ref{LEM:conjugation}, the measure $\nu_A$ is determined by its value on cylinder sets of the form $[(a_1,v_1)\cdots(a_n,v_n)]$. We define a measure $\mu = \desub(\nu_A)$ on $X_{\vartheta}$ by collapsing inflation words to letters via
\[
\mu([a_1 \cdots a_n]) = \sum_{v_1,\ldots,v_n} \nu_A([(a_1,v_1) \cdots (a_n,v_n)]),
\]
corresponding to a projection to the first coordinate. By construction, this gives an $S$-invariant probability measure that depends continuously on $\nu$.
Hence, the map $\Pi \colon \mc M \to \mc M, \, \nu \mapsto \desub(\nu_A)$ is a (weakly) continuous operator. The map $\Pi$ turns out to be surjective, but is generally not injective, since the  map $\desub$ can fail to be injective. We will use the random substitution action to construct appropriate inverse branches.

The action of $\vartheta_{\mathbf{P}}$ is given by a Markov kernel, replacing each letter independently. Under this transition, a given shift-invariant measure $\mu$ on $X_{\vartheta}$ is replaced by an $S_A$-invariant measure $\nu = \vartheta_{\mathbf{P}}(\mu)$ on $A = \vartheta(X_{\vartheta})$. Using the identification in Lemma~\ref{LEM:conjugation}, it can be explicitly expressed as
\begin{equation}
\label{EQ:substitution-on-measure}
(\vartheta_\mathbf{P}(\mu))([(a_1,v_1)\cdots(a_n,v_n)]) = \mu([a_1 \cdots a_n]) \mathbb{P}[\vartheta_{\mathbf{P}}(a_1 \cdots a_n) = v_1 \cdots v_n],
\end{equation}
for all legal blocks $a_1 \cdots a_n$ and $v_i \in \vartheta(a_i)$. Indeed, this represents an $S_A$-invariant probability measure on $A$.
Note that by construction, we have that $\desub(\vartheta_\mathbf{P}(\mu)) = \mu$, that is, $\desub$ is a left-inverse of $\vartheta_{\mathbf{P}}$.

\begin{lemma}
\label{LEM:substitution-image-ergodic}
The measure $\vartheta_{\mathbf{P}}(\mu)$ is $S_A$-ergodic if and only if $\mu$ is $S$-ergodic.
\end{lemma}

\begin{proof}
Recall that the ergodicity of $\mu$ is equivalent to 
\[
 \lim_{n \to \infty} \frac{1}{n} \sum_{k=0}^{n-1} \mu([x] \cap S^{-k}[y]) = \mu([x]) \mu([y]),
\]
for all $x,y \in \mc A^+$. Note that
\[
\mu([x] \cap S^{-k} [y]) = \sum_{z \in \mc A^{k-|x|}} \mu( [xzy]), 
\]
as soon as $k \geqslant |x|$. Suppose $u \in \vartheta(x)$ and $v \in \vartheta(y)$. Given a word $z \in \mc A^r$ and $w \in \vartheta(z)$, let $z_w = (z_1,w_1) \cdots (z_r,w_r)$, where $w = w_1 \cdots w_r$ is the unique decomposition with $w_i \in \vartheta(x_i)$ for all $1 \leqslant r$. 
With this notation, we obtain for $k \geqslant r$ that
\begin{align*}
\vartheta_{\mathbf{P}}(\mu)([x_u] \cap S_A^{-k}[y_v]) 
&= \sum_{z \in \mc A^{k-r}, w \in \vartheta(z)}
\vartheta_{\mathbf{P}}(\mu)([x_u z_w y_v])
\\ &= \sum_{z \in \mc A^{k-r}, w \in \vartheta(z)}
\mu([xzy]) \mathbb{P}[\vartheta_{\mathbf{P}}(xzy) = uwv]
\\ &= \mathbb{P}[\vartheta_{\mathbf{P}}(x) = u] \mathbb{P}[\vartheta_{\mathbf{P}}(y) = v] \, \mu([x] \cap S^{-k} [y])
\end{align*}
Hence, if $\mu$ is ergodic, we obtain for all $x_u$ and $y_v$ that
\begin{align*}
 \lim_{n \to \infty} \frac{1}{n} \sum_{k=0}^{n-1} \vartheta_{\mathbf{P}}(\mu)([x_u] \cap S_A^{-k}[y_v]) 
 &= \mathbb{P}[\vartheta_{\mathbf{P}}(x) = u] \mathbb{P}[\vartheta_{\mathbf{P}}(y) = v] \mu([x]) \mu([y])
\\ &= \vartheta_{\mathbf{P}}(\mu) ([x_u]) \, \vartheta_{\mathbf{P}}(\mu) ([y_v]),
\end{align*}
implying ergodicity of $\vartheta_{\mathbf{P}}(\mu)$. Conversely, if $\vartheta_{\mathbf{P}}(\mu)$ is ergodic, taking the sum over all $u \in \vartheta(x)$ and $v \in \vartheta(y)$ in the relation
\[
 \lim_{n \to \infty} \frac{1}{n} \sum_{k=0}^{n-1} \vartheta_{\mathbf{P}}(\mu)([x_u] \cap S_A^{-k}[y_v]) 
=\vartheta_{\mathbf{P}}(\mu) ([x_u]) \, \vartheta_{\mathbf{P}}(\mu) ([y_v])
\]
yields the ergodicity of $\mu$.
\end{proof}

We note that the integral of $r_A$ with respect to $\vartheta_{\mathbf{P}}(\mu)$ is given by
\[
\lambda_{\mu,\mathbf{P}} = \int_{A} r_A \dd \vartheta_{\mathbf{P}}(\mu)
= \sum_{(a,v) \in \mc B} \vartheta_{\mathbf{P}}(\mu)([(a,v)]) |v|
= \sum_{a \in \mc A} \mu([a]) \mathbb{E}[ |\vartheta_{\mathbf{P}}(a)| ].
\]
We call $\lambda_{\mu,\bm P}$ the \emph{normalisation constant} with respect to the defining data $\mu,\bm P$.
Using this normalisation factor, $\vartheta_{\mathbf{P}}(\mu)$ can be drawn back to an $S$-invariant measure $\trans_{\mathbf{P}}(\mu)$ in a canonical way. This is achieved by taking the average over all shifts prior to the first return to the set $A$, and normalising with the expected first return time, given by $\lambda_{\mu, \mathbf{P}}$.

\begin{definition}
Let $\mu \in \mc M$ and $\mathbf{P}$ a choice of probabilities for the random substitution $\vartheta$. The \emph{$\mathbf{P}$-transfer} of $\mu$ is the measure defined by
\[
\trans_{\mathbf{P}}(\mu)(f) 
= \lambda_{\mu, \mathbf{P}}^{-1} \vartheta_{\mathbf{P}}(\mu) \left(  \sum_{i=0}^{r_A -1} f \circ S^i \right),
\]
for all continuous functions $f$ on $X_{\vartheta}$.
\end{definition}
 Indeed, by Kac's formula, this is the only possible candidate for a shift-invariant measure with induced measure $\vartheta_{\mathbf{P}}(\mu)$ on $(A,S_A)$. That $\trans_{\mathbf{P}}(\mu)$ indeed defines a Borel measure follows easily from the Riesz--Markov--Kakutani representation theorem. Normalisation is checked by choosing $f \equiv 1$. Finally, $S$-invariance of $\trans_{\mathbf{P}}(\mu)$ follows in a straightforward manner from the $S_A$-invariance of $\vartheta_{\mathbf{P}}(\mu)$.

 \begin{lemma}
The measure $\trans_{\mathbf{P}}(\mu)$ is ergodic if and only if $\mu$ is ergodic.
 \end{lemma}

 \begin{proof}
     The ergodicity of $\trans_{\mathbf{P}}(\mu)$ is equivalent to the ergodicity of its induced measure $\vartheta_{\mathbf{P}}(\mu)$. By \Cref{LEM:substitution-image-ergodic} this is in turn equivalent to $\mu$ being ergodic.
 \end{proof}

 \begin{lemma}
 \label{LEM:continuous-transfer-operator}
For every choice of $\mathbf{P}$, the operator $\trans_{\mathbf{P}} \colon \mc M \to \mc M$ is continuous with respect to the topology of weak convergence.
\end{lemma}

\begin{proof}
The fact that the action of $\vartheta_\mathbf{P}$ is continuous follows in a straightforward manner from \eqref{EQ:substitution-on-measure}. 
Due to recognisability the first return map $r_A$ is continuous on $X_{\vartheta}$. Hence, for every continuous function $f$, the function $f_A = \sum_{i=0}^{r_A - 1} f\circ S^i$ is continuous on $A$. Therefore, $\vartheta_{\mathbf{P}}(\mu)(f_A)$ depends continuously on $\mu$ and the same holds for its normalisation $\lambda_{\mathbf{P},\mu}$, which is bounded away from $0$.
\end{proof}

\begin{definition}
 For $\nu \in \mc M$, let $\mathbf{P}(\nu)$ be the set of probability choices $\mathbf{P}$ satisfying \[
\nu_A([(a,v)]) =
\mathbf{P}_{v,a} \sum_{u \in \vartheta(a)} \nu_A([(a,u)]),
\]
for all $v \in \vartheta(a)$.
Given probability data $\mathbf{P}$, we let 
\[
\mc M[\mathbf{P}] = \{ \nu \in \mc M : \mathbf{P} \in \mathbf{P}(\nu) \}.
\]
\end{definition}

\begin{remark}
    Note that for all $\mathbf{P} \in \mathbf{P}(\nu)$ and $a \in \mc A$, the vector $\mathbf{P}^a = (\mathbf{P}_{v,a})_{v \in \vartheta(a)}$ is uniquely determined, as soon as $\Pi(\nu)([a])=\sum_{u \in \vartheta(a)} \nu_A([(a,u)]) > 0$, and it is completely arbitrary otherwise. In particular, $\mathbf{P}(\nu)$ is a singleton precisely if $\Pi(\nu)$ is non-vanishing on cylinders of length $1$.
\end{remark}

Recall that $R^\mu$ is the letter frequency vector defined by
$
R^\mu_a = \mu([a]),
$
for all $a \in \mc A$.
For notational convenience, we write $\mu \sim \nu$ for $\mu,\nu \in \mc M$ if $R^\mu = R^\nu$.
The induced measure of $\trans_{\mathbf{P}}(\mu)$ is given by $\vartheta_{\mathbf{P}}(\mu)$, which in turn is mapped back to $\mu$ under $\desub$. Hence, we observe that
 \[
 \Pi \trans_{\mathbf{P}}(\mu) = \mu,
 \]
 for all invariant probability measures $\mu$. From this, it follows that $\trans_{\mathbf{P}}(\mu) \in \mc M[\mathbf{P}] \cap \Pi^{-1}(\mu)$. All measures in this set have the same letter frequencies, as we show below. 
 Given probability data $\mathbf{P}$, recall that $M(\mathbf{P})$ is the substitution matrix of $\vartheta_{\mathbf{P}}$. 

\begin{lemma}
\label{LEM:letter-frequency-uniformity}
Given $\mu \in \mc M$ and some probability data $\mathbf{P}$, all measures $\nu \in \mc M[\mathbf{P}] \cap \Pi^{-1}(\mu)$ share the same letter frequency vector $R^\nu$, given by 
\[
R^\nu = \lambda_{\mu,\mathbf{P}}^{-1} M(\mathbf{P}) R^\mu.
\]
In particular, this applies to $\nu = \trans_{\mathbf{P}}(\mu)$.
Also, $\mu \sim \mu'$ implies that $\trans_{\mathbf{P}}(\mu) \sim \trans_{\mathbf{P}}(\mu') $ for all $\mathbf{P}$.
\end{lemma}

\begin{proof}
First, note that $\mu = \Pi(\nu) = \desub(\nu_A)$ satisfies
\[
\mu([b]) = \sum_{v \in \vartheta(b)} \nu_A([(b,v)]),
\]
and hence $\nu_A([(b,v)]) = \mathbf{P}_{v,b} \mu([b])$, by the defining relation for $\mathbf{P} \in \mathbf{P}(\nu)$. Hence, for $A = \vartheta(X_{\vartheta})$,
\[
\nu(A)^{-1} = \int_{A} r_A \dd \nu_A
= \sum_{(a,v) \in \mc B} \nu_A([(a,v)]) |v|
= \sum_{a \in \mc A} \mu([a]) \sum_{v \in \vartheta(a)} \mathbf{P}_{v,a} |v| = \lambda_{\mu,\mathbf{P}}.
\]
Combining these observations with Kac's formula yields
\[
\nu([a]) = \frac{1}{\lambda_{\mu,\mathbf{P}}} \sum_{(b,v) \in \mc B} \nu_A([b,v]) \phi_a(v)
= \frac{1}{\lambda_{\mu,\mathbf{P}}} \sum_{b \in \mc A} \mu([b])
\sum_{v \in \vartheta(a)} \mathbf{P}_{v,b} \phi_a(v),
\]
for all $a \in \mc A$. Since the entries of $M(\mathbf{P})$ are given by
\[
M(\mathbf{P})_{ab} = \mathbb{E}_{\mathbf{P}}\phi_a(\vartheta(b))
= \sum_{v \in \vartheta(b)} \mathbf{P}_{v,b} \phi_a(v),
\]
this shows the stated formula for $R^\nu$. For the final claim, we apply this to $\nu = \trans_{\mathbf{P}}(\mu)$ and observe that the letter frequencies of this measure depend only on $\mathbf{P}$ and $R^\mu$.
\end{proof}

The updating rule for the frequency vector under $\trans_{\mathbf{P}}$ takes an even easier form in the geometric setting. Recall that the interval proportion vector $\pi^\mu$ is given by
$
\pi^\mu_a = \Length_a R^\mu_a/\Length R^\mu.
$
We emphasise that the relation between $R^\mu$ and $\pi^\mu$ is one-to-one, since
$
R^\mu_a = \Length^{-1}_a \pi^\mu_a/(\sum_{b \in \mc A} \Length^{-1}_b \pi^\mu_b).
$

\begin{corollary}
\label{COR:interval-proportion-transformation}
    The interval proportion vectors of $\nu = \trans_{\mathbf{P}}(\mu)$ and $\mu$ are related by
    \[
    \pi^\nu = Q(\mathbf{P}) \pi^\mu,
    \]
    where $Q(\mathbf{P})$ is the geometric substitution matrix of $\vartheta_{\mathbf{P}}$.
\end{corollary}

\begin{proof}
First note that, due to \Cref{LEM:letter-frequency-uniformity},
\begin{equation}
\label{EQ:inflation-factors-relation}
\lambda_{\mu,\mathbf{P}} \Length R^\nu =  \Length M(\mathbf{P}) R^\mu = \lambda \Length R^\mu.
\end{equation}
Hence, 
\[
\pi^\nu_a = \frac{\Length_a R^\nu_a}{\Length R^\nu} 
= \frac{ \Length_a}{ \lambda_{\mu,\mathbf{P}} \Length R^\nu} \sum_{b \in \mc A} M(\mathbf{P})_{ab} R^\mu_b
= \sum_{b \in \mc A} \frac{\Length_a M(\mathbf{P})_{ab}}{\lambda \Length_b} \frac{\Length_b R^\mu_b}{\Length R^\mu} 
= \sum_{b \in \mc A} Q(\mathbf{P})_{ab} \pi^\mu_b,
\]
as claimed.
\end{proof}

Recall that $\Pi$ is a left-inverse of $\trans_{\mathbf{P}}$, irrespective of $\mathbf{P}$. 
The special role of $\trans_{\mathbf{P}}$ as an inverse branch of $\Pi$ is that it maximises the entropy of all measures with a given inflation word distribution. 

 \begin{prop}
 \label{LEM:entropy-maximising-inverse-branch}
For each $\mu \in \mc M$ and probability data $\mathbf{P}$, the measure $\trans_{\mathbf{P}}(\mu)$ is the unique measure of maximal geometric entropy in $\Pi^{-1}(\mu) \cap \mc M[\mathbf{P}]$, and satisfies
\[
 \hg_{\trans_{\mathbf{P}}(\mu)} = \frac{1}{\lambda} \left( \hg_\mu + \frac{ H_{\mathbf{P}} R^\mu }{\Length R^\mu}\right).
 \]
 \end{prop}

 \begin{proof}
For $\nu \in \Pi^{-1}(\mu) \cap \mc M[\mathbf{P}]$, let $\nu_A$ be the induced measure on $A=\vartheta(X_{\vartheta})$. By Abramov's formula, we have
$h_{\nu} = \nu(A) h_{\nu_A} = \lambda_{\mu,\mathbf{P}}^{-1} h_{\nu_A}$ and, using \Cref{LEM:letter-frequency-uniformity} as in \eqref{EQ:inflation-factors-relation},
\[
\Length R^\nu = \frac{\lambda}{\lambda_{\mu,\mathbf{P}}} \Length R^\mu,
\]
implying 
\[
\hg_{\nu} = \frac{h_{\nu_A}}{\lambda \Length R^\mu}.
\]
Since the normalisation is fixed, $\nu$ achieves maximal  geometric entropy precisely if $h_{\nu_A}$ is maximal.
The measure $\nu_A$ naturally induces a distribution $\nu_A^n$ on $\mc B^n$ by setting $\nu_A^n(u) = \nu_A([u])$ for each $u \in \mc B^n$. With this notation, and the identity map $\operatorname{Id}_n \colon u \mapsto u$ on $\mc B^n$ we have
\[
h_{\nu_A} = \lim_{n \to \infty} \frac{1}{n} \Ent_{\nu_A^n}(\operatorname{Id}_n)
= \inf_{n \in \N} \frac{1}{n} \Ent_{\nu_A^n}(\operatorname{Id}_n).
\]
We define several random variables on $\mc B^n$, which are defined on $u=(x_i,v_i)_{i=1}^n$ via $X_i(u)=x_i$, $V_i(u) = v_i$ and $X_{[j,k]}(u) = x_j \cdots x_k$ for $1\leqslant j\leqslant k \leqslant n$. 
If not specified otherwise, we fix the distribution $\rho = \nu_A^n$ on $\mc B^n$ and compute entropies with respect to this measure.
Note that $u$ is completely determined by the values of $X_{[1,n]}(u)$ and $(V_i(u))_{i=1}^n$, so we obtain
\[
\Ent_{\nu_A^n}(\operatorname{Id}_n) =  \Ent(V_1,\ldots,V_n,X_{[1,n]}) = \Ent(X_{[1,n]}) + \Ent(V_1,\ldots,V_n|X_{[1,n]}).
\]
The distribution of $X_{[1,n]}$ is given by
\begin{equation}
\label{EQ:rho-on-X-measure}
\rho(\{X_{[1,n]} = x_1 \cdots x_n\}) = \sum_{v_1,\ldots,v_n} \nu_A([(x_1,v_1)\cdots(x_n,v_n)]) = \mu([x_1 \cdots x_n]), 
\end{equation}
using that $\nu \in \Pi^{-1}(\mu)$ implies $\desub(\nu_A) = \mu$ in the last step. Hence, we obtain that
\[
\lim_{n \to \infty} \frac{1}{n} \Ent(X_{[1,n]}) = \inf_{n \to \infty} \frac{1}{n} \Ent(X_{[1,n]}) = h_{\mu},
\]
which is uniform for all $\nu \in \Pi^{-1}(\mu)$. 
We thus focus on the term $\Ent(V_1,\ldots, V_n|X_{[1,n]})$. By standard properties of conditional entropy, we have
\begin{equation}
\label{EQ:entropy-V-split}
\Ent(V_1,\ldots,V_n|X_{[1,n]}) \leqslant \sum_{i=1}^n \Ent(V_i|X_{[1,n]})
\end{equation}
and for each $1\leqslant i \leqslant n$,
\begin{equation}
\label{EQ:entropy-X-split}
\Ent(V_i|X_{[1,n]}) \leqslant \Ent(V_i|X_i).
\end{equation}
The shift-invariance of $\nu_A$ implies that $\Ent(V_i|X_i) = \Ent(V_1|X_1)$ for all $1\leqslant i \leqslant n$, given by
\[
\Ent(V_1|X_1) = \sum_{a \in \mc A}\rho(\{X_1 = a\}) \Ent_{\rho_{\{ X_1 = a\}}}(V_1).
\]
Similarly to before, we have that $\rho(\{ X_1 = a\}) = \mu([a])$ and, provided that $\mu([a])>0$,
\[
\rho_{\{ X_1 =a \}}(\{V_1 = v \}) = \frac{\nu_A([(a,v)])}{\sum_{u \in \vartheta(a)} \nu_A([(a,u)])} = \mathbf{P}_{v,a},
\]
where the last equality follows from $\nu \in \mc M[\mathbf{P}]$. 
Hence, 
$
\Ent_{\rho_{\{ X_1 = a\}}}(V_1) = \Ent(\vartheta_{\mathbf{P}}(a))
$, for such $a \in \mc A$.
If $\mu([a]) = 0$, the set $\{ X_1 = a\}$ has vanishing measure and does not contribute to $\Ent(V_1|X_1)$. Hence,
\[
\Ent(V_1|X_1) = \sum_{a \in \mc A} \mu([a]) \Ent_{\mathbb{P}}(\vartheta_{\mathbf{P}}(a)) = H_{\mathbf{P}} R^\mu.
\]
In summary, we obtain that $\Ent(V_1,\ldots,V_n|X_{[1,n]}) \leqslant n H_{\mathbf{P}} R^\mu$ and thereby
\begin{equation}
\label{EQ:induced-measure-entropy-estimate}
h_{\nu_A} = \lim_{n \to \infty} \frac{1}{n} \bigl( \Ent(X_{[1,n]}) + \Ent(V_1, \ldots, V_n|X_{[1,n]}) \bigr) 
\leqslant h_{\mu} + H_{\mathbf{P}} R^\mu.
\end{equation}
We claim that this inequality is an equality precisely if $\Ent(V_1,\ldots,V_n|X_{[1,n]}) = n H_{\mathbf{P}} R^\mu$ for all $n \in \N$.
That this condition is sufficient for equality in the entropy expression is apparent. Conversely, assume that for some $n \in \N$, we have $\Ent(V_1,\ldots,V_n|X_{[1,n]}) < n H_{\mathbf{P}} R^\mu$. Then, there is $\varepsilon > 0$ such that
$\Ent(V_1,\ldots,V_n|X_{[1,n]}) = n H_{\mathbf{P}} R^\mu - \varepsilon$.
For each $m\in \N$, we obtain that
\begin{align*}
\Ent(V_1,\ldots,V_{mn}|X_{[1,mn]})
&\leqslant \sum_{i=0}^{m-1} \Ent(V_{in+1},\ldots,V_{(i+1)n}|X_{[1,mn]}) 
\\ &\leqslant \sum_{i=0}^{m-1} \Ent(V_{in+1},\ldots,V_{(i+1)n}|X_{[in+1,(i+1)n]}) 
\\ & \leqslant m \Ent(V_1,\ldots,V_n|X_{[1,n]})
\leqslant mn H_{\mathbf{P}} R^\mu- m\varepsilon,
\end{align*}
using the invariance of $\nu_A$ in the penultimate step. This implies
\[
h_{\nu_A} = \lim_{m \to \infty} \frac{1}{mn} \bigl( \Ent(X_{[1,mn]}) + \Ent(V_1, \ldots, V_{mn}|X_{[1,mn]}) \bigr) 
\leqslant h_{\mu} + H_{\mathbf{P}} R^\mu - \frac{\varepsilon}{n},
\]
and the inequality in \eqref{EQ:induced-measure-entropy-estimate} is indeed strict. Note that $\Ent(V_1,\ldots,V_n|X_{[1,n]}) = n H_{\mathbf{P}} R^\mu$ if and only if we have equality in both \eqref{EQ:entropy-V-split} and $\eqref{EQ:entropy-X-split}$.
Equality in \eqref{EQ:entropy-V-split} holds if and only if for every set of the form $S_n = \{ X_{[1,n]} = x_1 \cdots x_n \}$ with positive $\rho$ measure, the random variables $V_1,\ldots,V_n$ are independent with respect to the induced measure $\rho_{S_n}$. Phrased differently,
\begin{equation}
\label{EQ:conditional-independence}
\rho_{\{ X_{[1,n]} = x_1 \cdots x_n \}} (\{V_1 = v_1,\ldots,V_n=v_n \}) 
= \prod_{i=1}^n \rho_{\{ X_{[1,n]} = x_1 \cdots x_n \}} (\{V_i = v_i \}).
\end{equation}
On the other hand, equality in \eqref{EQ:entropy-X-split} means that, given $X_1$, the realisation of $V_1$ is independent of $X_2,\ldots,X_n$. That is, for every realisation $x_1 \cdots x_n$ of positive measure, we have that
\begin{equation}
\label{EQ:conditional-refinement}
\rho_{\{ X_{[1,n]} = x_1 \cdots x_n \}}(\{V_1 = v_1\})
= \rho_{\{ X_1 = x_1 \}} (\{ V_1 = v_1 \}) = \mathbf{P}_{v_1,x_1}.
\end{equation}
Recalling the normalisation in \eqref{EQ:rho-on-X-measure}, equality in \eqref{EQ:induced-measure-entropy-estimate} hence requires that
\begin{equation}
\label{EQ:rho-relation-to-P}
\prod_{i=1}^n \mathbf{P}_{v_i,x_i} = \frac{\rho((x_1,v_1)\cdots (x_n,v_n))}{\mu([x_1\cdots x_n])}.
\end{equation}
This is equivalent to
\[
\nu_A([(x_1,v_1)\cdots (x_n,v_n)]) = \mu([x_1 \cdots x_n]) \mathbb{P}[\vartheta_{\mathbf{P}}(x_1 \cdots x_n) = v_1 \cdots v_n],
\]
which remains true if $\mu([x_1 \cdots x_n]) = 0$ because $\mu = D(\nu_A)$.
We therefore find equivalence to $\nu_A = \vartheta_{\mathbf{P}}(\mu)$, which is in turn equivalent to $\nu = \trans_{\mathbf{P}}(\mu)$.
Conversely, it is straightforward to verify that the distribution $\rho$ fixed by \eqref{EQ:rho-relation-to-P} indeed satisfies both \eqref{EQ:conditional-independence} and \eqref{EQ:conditional-refinement}. 
From this, we conclude that equality in \eqref{EQ:induced-measure-entropy-estimate} holds if and only if $\nu = \trans_{\mathbf{P}}(\mu)$.

In this case, we obtain for the geometric entropy of $\nu = \trans_{\mathbf{P}}(\mu)$ the explicit expression
\[
\hg_\nu = \frac{h_{\nu_A}}{\lambda \Length R^\mu}
= \frac{1}{\lambda} \left( \hg_\mu + \frac{H_{\mathbf{P}} R^\mu}{\Length R^\mu}  \right),
\]
which is precisely the claimed relation.
 \end{proof}

 \begin{corollary}
 \label{COR:entropy-maximising-uniform}
For each $\mu \in \mc M$, the unique measure of maximal geometric entropy in $\Pi^{-1}(\mu)$ is given by $\trans_{\mathbf{P}}(\mu)$, where $(\mathbf{P}_{v,a} )_{v \in \vartheta(a)}$ is the uniform distribution for each $a$, and satisfies
\[
\hg_{\trans_{\mathbf{P}} (\mu)} = \frac{1}{\lambda} \left( \hg_{\mu} + \frac{1}{\Length R^\mu} \sum_{a \in \mc A} \mu([a]) \log \# \vartheta(a) \right).
\]
 \end{corollary}

\begin{proof}
Note that $\Ent(\vartheta_{\mathbf{P}}(a)) = \log \# \vartheta(a)$ if and only if $\mathbf{P}^a =  (\mathbf{P}_{v,a} )_{v \in \vartheta(a)}$ is the uniform distribution on $\vartheta(a)$. Hence,
\[
H_{\mathbf{P}} R^\mu =
\sum_{a \in \mc A} \mu([a])\Ent(\vartheta_{\mathbf{P}}(a))
= \sum_{a \in \mc A} \mu([a]) \log \# \vartheta(a),
\]
if and only if $\mathbf{P}^a$ is uniform for all $a$ with $\mu([a]) \geqslant 0$. We observe that for $\mu([a]) = 0$ the choice of $\mathbf{P}^a$ is immaterial for $\trans_{\mathbf{P}}(\mu)$ and can hence be chosen uniform without altering the measure.
The claim therefore follows from \Cref{LEM:entropy-maximising-inverse-branch} by decomposing $\Pi^{-1}(\mu) = \cup_{\mathbf{P}} (\Pi^{-1}(\mu) \cap \mc M[\mathbf{P}])$.
\end{proof}

\begin{lemma}
\label{LEM:T-consistency}
We have $\trans_{\mathbf{P} \mathbf{P}'} = \trans_{\mathbf{P}} \circ \trans_{\mathbf{P}'}$.
\end{lemma}

\begin{proof}
Given $\mu \in \mc M$, we need to prove that $\trans_{\mathbf{P} \mathbf{P}'}(\mu) = \trans_{\mathbf{P}}(\nu)$, where $\nu = \trans_{\mathbf{P}'}(\mu)$. 
Since both measures are $S$-invariant, it suffices to show that they coincide on $A_2 = \vartheta^2(X_{\vartheta}) \subset \vartheta(X_{\vartheta})$.
Every cylinder on $A_2$ is of the form 
\[
C_2(x,w) = [(x_1,w_1^{(2)}) \cdots (x_n,w_n^{(2)})],
\]
where $x = x_1 \cdots x_n \in \mc L_n$ and $w = w_1^{(2)} \cdots w_n^{(2)}$ is the unique decomposition of $w$ such that $w_i^{(2)} \in \vartheta^2(x_i)$ for all $1 \leqslant i \leqslant n$. By recognisability, there is a unique word $v$ with $v \in \vartheta(x)$ and $w \in \vartheta(v)$. Assuming $v=v_1 \cdots v_n$, with $v_i \in \vartheta(x_i)$ for all $1\leqslant i \leqslant n$, there is hence a unique preimage of $C_2(x,w)$ under $\vartheta$ in $A = \vartheta(X_{\vartheta})$, given by
\[
C_1(x,v) = [(x_1,v_1) \cdots (x_n,v_n)].
\]
Since $C_2(x,w) \in A_2 \subset A$, we obtain
\begin{align*}
\trans_{\mathbf{P}}(\nu)(C_2(x,w))
& = \frac{1}{\lambda_{\nu,\mathbf{P}}} \vartheta_{\mathbf{P}}(\nu)(C_2(x,w))
= \frac{1}{\lambda_{\nu,\mathbf{P}}} \nu(C_1(x,v)) \mathbb{P}[\vartheta_{\mathbf{P}}(v) = w]
\\ & = \frac{1}{\lambda_{\nu,\mathbf{P}} \lambda_{\mu,\mathbf{P}'}} \mu([x]) \mathbb{P}[\vartheta_{\mathbf{P}'}(x) = v] \mathbb{P}[\vartheta_{\mathbf{P}}(v) = w]
\\ &= \frac{1}{\lambda_{\nu,\mathbf{P}} \lambda_{\mu,\mathbf{P}'}} \mu([x]) \mathbb{P}[\vartheta^2_{\mathbf{P} \mathbf{P}'}(x) = w]
=  \frac{\lambda_{\mu,\mathbf{P} \mathbf{P}'}}{\lambda_{\nu,\mathbf{P}} \lambda_{\mu,\mathbf{P}'}} \trans_{\mathbf{P} \mathbf{P}'}(\mu)(C_2(x,w))
.
\end{align*}
Hence, it remains to show that $\lambda_{\nu,\mathbf{P}} \lambda_{\mu,\mathbf{P}'} = \lambda_{\mu,\mathbf{P} \mathbf{P}'}$.
Indeed, recalling that $\nu = \trans_{\mathbf{P}'}(\mu)$ and writing $\mathds{1}$ for the vector with constant entries $1$, we obtain by \Cref{LEM:letter-frequency-uniformity} that
\[
\lambda_{\nu,\mathbf{P}} \lambda_{\mu,\mathbf{P}'}
= \lambda_{\mu,\mathbf{P}'} \mathds{1}^T M(\mathbf{P}) R^{\trans_{\mathbf{P}'}(\mu)}
= \mathds{1}^T M(\mathbf{P}) M(\mathbf{P}') R^\mu
= \mathds{1}^T M(\mathbf{P} \mathbf{P}') R^\mu
= \lambda_{\mu,\mathbf{P} \mathbf{P}'}.
\]
This shows equality of  $\trans_{\mathbf{P} \mathbf{P}'}(\mu) $ and $ \trans_{\mathbf{P}} \circ \trans_{\mathbf{P}'}(\mu)$ on $A_2$, implying equality on the whole space due to shift invariance.
\end{proof}

\begin{lemma}\label{LEM:frequency-fixed-measure}
The frequency measure $\mu_{\mathbf{P}}$ is the unique fixed point of $\trans_{\mathbf{P}}$. That is, $\mu_{\mathbf{P}}$ is the unique invariant probability measure $\mu$ with $\trans_{\mathbf{P}}(\mu) = \mu$.
\end{lemma}
\begin{proof}
Let $\nu$ be an arbitrary invariant probability measure. We will show that $T=\trans_{\mathbf{P}}$ satisfies $\lim_{n \to \infty} T^n(\nu) = \mu_{\mathbf{P}}$. Using the continuity of $T$, this implies that $T(\mu_{\mathbf{P}}) = \mu_{\mathbf{P}}$. Conversely, if $\nu = T(\nu)$ this directly implies that $\nu = \mu_{\mathbf{P}}$. By \Cref{LEM:T-consistency}, we know that $T^n = T_{\mathbf{P}^n}$. Set $A_n = \vartheta^n(X_{\vartheta})$, and given $u \in \mc A^+$ let $\chi_{[u]}$ be the indicator function of $[u]$. Recall that
\[
T^n(\nu)([u]) = \frac{1}{\lambda_{\mathbf{P}^n,\nu}} \vartheta^n_{\mathbf{P}^n}(\nu)(f_u),
\quad
f_u = \sum_{i=0}^{r_{A_n} - 1} \chi_{[u]} \circ S^i.
\]
Since the ratio between the geometric and symbolic length of a word is bounded, we directly obtain that $\lim_{n \to \infty}\lambda_{\mathbf{P}^n,\nu} = \infty$. Further, note that on $[(a,v)]$ with $v \in \vartheta^n(a)$, we can estimate $f_u$ via $|v|_u \leqslant f_u \leqslant |v|_u + |u|$, and therefore we obtain
\begin{align*}
\vartheta^n_{\mathbf{P}^n}(f_u) & = \sum_{(a,v) \in \mc B^n} \vartheta^n_{\mathbf{P}^n}(\nu)([a,v]) |v|_u + O(|u|)
= \sum_{a \in \mc A} \nu([a]) \mathbb{E}[|\vartheta^n_{\mathbf{P}^n}(a)|_u] + O(|u|).
\\ &\approx \sum_{a \in \mc A} \nu([a])  \mu_{\mathbf{P}}([u]) \mathbb{E}[|\vartheta^n_{\mathbf{P}^n}(a)|]
= \mu_{\mathbf{P}}([u]) \lambda_{\mathbf{P}^n,\nu},
\end{align*}
yielding the desired convergence $\lim_{n \to \infty} T^n(\nu)([u]) = \mu_{\mathbf{P}}([u])$.
\end{proof}

  Note that for the class of random substitutions considered here, the known formula for the entropy of the frequency measure is a direct consequence of \Cref{LEM:frequency-fixed-measure} and \Cref{LEM:entropy-maximising-inverse-branch}.

\section{Inverse limit measures}
\label{SEC:inverse-limit-measures}

\subsection{Construction and entropy maximisation}

We start from a sequence of probability choices $\mc P = (\mathbf{P}_n)_{n \in \N}$ for the random substitution $\vartheta$. Our aim is to construct a measure that represents the word frequencies in $\vartheta_{\mathbf{P}_1} \circ \cdots \circ \vartheta_{\mathbf{P}_n}$ as $n \to \infty$, provided they are well defined. This is similar to the construction of invariant measures for S-adic systems.
We say that a sequence of $S$-invariant probability measures $(\mu_n)_{n \in \N}$ on $X_{\vartheta}$ is \emph{$(\vartheta, \mc P)$-adapted} if $\mu_n = \trans_{\mathbf{P}_n}(\mu_{n+1})$ for all $n \in \N$.
Likewise, we call $\mu$ a \emph{$(\vartheta,\mc P)$ inverse limit measure} if $\mu = \mu_1$ for some $(\vartheta,\mc P)$-adapted sequence $(\mu_n)_{n \in \N}$.
In this case, we have $\mu_{n+1} = \Pi^n(\mu)$ for all $n\in\N$.

\begin{remark}
\label{REM:frequency_measures-as-inverse-limits}
If $\mathbf{P}_n = \mathbf{P}$ for all $n \in \N$ is a constant sequence, it follows from the proof of \Cref{LEM:frequency-fixed-measure} that $\mu_\mathbf{P}$ is the unique $(\vartheta,\mc P)$ inverse limit measure. By considering higher powers of $\vartheta$, we also obtain a unique inverse limit measure if $\mc P$ is periodic and by extension if $\mc P$ is eventually periodic. 
\end{remark}

In general, the uniqueness of inverse limit measures is a subtle issue, but existence follows routinely via compactness in our setting.

\begin{lemma}
\label{LEM:existence-of-inverse-limit-measures}
For each sequence of probability choices $\mc P$ for $\vartheta$, there exists a $(\vartheta,\mc P)$ inverse limit measure.
\end{lemma}

\begin{proof}
Let $\nu \in \mc M$ be arbitrary. For all $n \in \N$ and $1\leqslant i \leqslant n$, let 
\[
\mu^{(n)}_i = \trans_{\mathbf{P}_i} \circ \cdots \circ \trans_{\mathbf{P}_n} (\nu).
\]
By construction, the finite sequence $(\mu_i^{(n)})_{i=1}^n$ satisfies 
\begin{equation} 
\label{EQ:approximants-transfer}
\mu^{(n)}_{i} = \trans_{\mathbf{P}_{i}}(\mu^{(n)}_{i+1})
\end{equation}
for all $1 \leqslant i \leqslant n-1$.
By compactness of $\mc M$, for each $i \in \N$, the sequence $(\mu_i^{(n)})_{n\geqslant i}$ has an accumulation point. Using a diagonal argument, we can choose an increasing subsequence $(n_j)_{j \in \N}$ of natural numbers such that for all $i \in \N$, we have $\lim_{j \to \infty} \mu^{(n_j)}_i = \mu_i$ for some $\mu_i \in \mc M$. The relation $\mu_{i+1} = \trans_{\mathbf{P}_i}(\mu_i)$ follows from \eqref{EQ:approximants-transfer} and the continuity of $\trans_{\mathbf{P}_i}$. This shows that $(\mu_i)_{i \in \N}$ is $(\vartheta,\mc P)$-adapted and hence $\mu_1$ is a $\mc P$ inverse limit measure.
\end{proof}

We show that $(\vartheta,\mc P)$ inverse limit measures are abundant enough to produce all possible letter frequencies while maximising the corresponding entropy.

\begin{theorem}
\label{THM:fixed-frequency-inverse-limit}
For each $\nu \in \mc M$, there exists a $\mc P$ and a $(\vartheta,\mc P)$ inverse limit measure $\mu$ with $\mu \sim \nu$ and $\hg_\mu \geqslant \hg_\nu$. If $\nu$ is not an inverse limit measure, we can choose $\mu$ such that $\hg_{\mu} > \hg_\nu$.
\end{theorem}

\begin{proof}
We start from an arbitrary measure $\nu = \mu^1_1 \in \mc M$ and inductively show the existence of a family $\{ \mu^n_i : n \in \N, 1\leqslant i \leqslant n  \}$ and $\mc P = (\mathbf{P}_i)_{i \in \N}$ with the following properties:
\begin{enumerate}
\item $\mu^{i+1}_{i+1} = \Pi(\mu^i_i)$ for all $i \in \N$;
\item $\mathbf{P}_i = \mathbf{P}(\mu^i_i)$ for all $i \in \N$;
\item $\mu^n_{i} = \trans_{\mathbf{P}_i} (\mu^n_{i+1}) $ for all $1 \leqslant i < n$;
\item $\mu^n_i \sim \mu^m_i$ for all $i \in \N$ and $n,m \geqslant i$;
\item $\hg_{\mu^{n}_i} \geqslant \hg_{\mu^{n-1}_i}$, with equality if and only if $\mu^{n}_i = \mu^{n-1}_i$. 
\end{enumerate}

Assume that, for some $N \in \N$ and all $1 \leqslant i \leqslant n \leqslant N$, the measures $\mu^n_i$ are well defined and fulfil the properties above. For $N=1$ this clearly holds. We perform the inductive step by showing that the same holds up to $N+1$. The first three properties are simply definitions, fixing the value of $\mathbf{P}_{N+1}$, as well as $\mu^{N+1}_{N+1} = \Pi(\mu^N_N)$, and $\mu^{N+1}_{i} = \trans_{\mathbf{P}_i} (\mu^{N+1}_{i+1}) $ for all $1 \leqslant i \leqslant N$. 
For the fourth property, it suffices to show that $\mu_i^N \sim \mu_i^{N+1}$ for all $1\leqslant i \leqslant N$. 
For $i = N$, this follows from the fact that $\mu_{N}^{N+1} = \trans_{\mathbf{P}_N}(\mu^{N+1}_{N+1})$ and $\mu^N_N$ are both in $\Pi^{-1}(\mu^{N+1}_{N+1}) \cap \mc M[\mathbf{P}_N]$. Indeed, by \Cref{LEM:letter-frequency-uniformity}, this implies that $\mu_N^{N+1} \sim \mu^N_N$. By the last statement in \Cref{LEM:letter-frequency-uniformity}, this also implies that
\[
\mu_i^N = \trans_{\mathbf{P}_i} \cdots \trans_{\mathbf{P}_{N-1}} \mu^N_N
\sim 
\trans_{\mathbf{P}_i} \cdots \trans_{\mathbf{P}_{N-1}} \mu^{N+1}_N
= \mu_i^{N+1},
\]
for all $i \leqslant N-1$, which completes the proof of the fourth property.
For the last property it remains to show that $\hg_{\mu_i^{N+1}} \geqslant \hg_{\mu_i^{N}}$ with equality if and only if the measures are equal. For $i =N$, this follows from the fact that $\mu_{N}^{N+1} = \trans_{\mathbf{P}_N}(\mu^{N+1}_{N+1})$ is the unique measure of maximal geometric entropy in $\Pi^{-1}(\mu^{N+1}_{N+1}) \cap \mc M[\mathbf{P}_N]$ by \Cref{LEM:entropy-maximising-inverse-branch}. Recall that by \Cref{LEM:entropy-maximising-inverse-branch},
\[
\hg_{\trans_{\mathbf{P}}(\mu)} = \frac{1}{\lambda} \left( \hg_\mu + \frac{H_{\mathbf{P}} R^\mu}{\Length R^\mu} \right)
\]
is strictly increasing in the entropy of $\mu$ as long as the letter frequencies and $\mathbf{P}$ remain fixed. Since $\mu_N^{N} \sim \mu^{N+1}_N$, this also shows that the entropy of $\mu_{N-1}^{N+1} = \trans_{\mathbf{P}_{N-1}}(\mu^{N+1}_N)$ is at least the entropy of $\mu_{N-1}^N = \trans_{\mathbf{P}_{N-1}}(\mu^N_N)$ and if the measures are not equal, then $\mu^N_N \neq \mu_N^{N+1}$, implying that the inequality of entropies is strict. Inductively, the same holds for the entropies of $\mu^{N+1}_i$ and $\mu^N_i$, finishing the proof of the last property.

In summary, we obtain for each $i \in \N$ a sequence $(\mu_i^n)_{n \geqslant i}$ of measures with identical letter frequencies and increasing geometric entropy. There exists an appropriate (diagonal) subsequence $(n_j)_{j \in \N}$ with respect to which each of these measure sequences converges and we set $\mu_i = \lim_{j\to \infty} \mu_i^{n_j}$ for all $i \in \N$. By continuity of the transfer operators and the third property, we obtain that $\mu_i = \trans_{\mathbf{P}_i}(\mu_{i+1})$ and in particular, $\mu_1$ is a $(\vartheta, \mc P)$ inverse limit measure. Since equality of letter frequencies is preserved under weak convergence, we also have $\mu_1 \sim \mu_1^1$.
Furthermore, by the upper semi-continuity of entropy, we also obtain that
\[
\hg_{\mu_1} \geqslant \sup_{n \in \N} \hg_{\mu_1^n} \geqslant \hg_{\mu_1^1},
\]
since the sequence $(\hg_{\mu_1^n})_{n \in \N}$ is increasing. We can only have equality if the entropy sequence is constant, implying that the measure sequence is constant. In this case $\mu^1 = \mu^1_1$, implying that the starting measure was already an inverse limit measure.
\end{proof}

All of the above can be generalised to higher powers of the random substitution $\vartheta$. In particular, we have $(\vartheta^n, \mc P)$ limit measures for all sequences $\mc P$ of probability choices for $\vartheta^n$ and $n \in \N$. 

For each $n \in \N$ and $\nu \in \mc M$, we regard the induced measure $\nu_{A_n}$ on $A_n = \vartheta^n(X_{\vartheta})$ as a measure on $\mc B_n^\Z$, where $\mc B_n = \{(a,v): a \in \mc A, v\in\vartheta^n(a) \}$.
In particular, we write $\nu_{A_n} \sim \mu_{A_n}$ if $\nu_{A_n}([(a,v)]) = \mu_{A_n}([(a,v)])$ for all $a \in \mc A$ and $v \in \vartheta^n(a)$.

\begin{definition}
For each $\nu \in \mc M$ and $n \in \N$ with $A_n = \vartheta^n(X_{\vartheta})$, we set
\[
\mc M[\nu,n] = \{ \mu \in \mc M: \mu_{A_n} \sim \nu_{A_n} \}.
\]
In particular, $\mc M[\nu,0] = \{ \mu \in \mc M: \mu \sim \nu \}$.
\end{definition}

Our next aim is to show that these inverse limit measures are dense in the space of all shift-invariant probability measures. As a first step, we show that every $\mc M[\nu,n]$ contains an inverse limit measure that maximises the geometric entropy.

\begin{lemma}
\label{LEM:inverse-limits-for-inflations-word-frequencies}
For each $n \in \N$ and $\nu \in \mc M$, every measure of maximal geometric entropy in $\mc M[\nu,n]$ is a $(\vartheta^n,\mc P)$ inverse limit measure for some sequence of probability choices $\mc P$ for $\vartheta^n$.
\end{lemma}

\begin{proof}
Since $\Pi^n$ collapses level-$n$ inflation words to letters, it is straightforward to verify that for each $\nu'\in \mc M[\nu,n]$, we have that $\Pi^n(\nu') \sim \Pi^n(\nu)$. That is, $\nu' \in \Pi^{-n}(\mu)$ for some $\mu \sim \Pi^n(\nu)$ and we can decompose
\[
\mc M[\nu,n] = \bigcup_{\mu \colon \mu \sim \Pi^n(\nu)} \mc M[\nu,n] \cap \Pi^{-n}(\mu). 
\]
Writing $\mathbf{P}_n$ for (some choice of) the probability data of $\nu$ on level-$n$ inflation words, we easily verify that $M[\nu,n] \cap \Pi^{-n}(\mu) = \mc M[\mathbf{P}_n] \cap \Pi^{-n}(\mu)$. Applying \Cref{LEM:entropy-maximising-inverse-branch} to $\vartheta^n$, we obtain that the unique measure of maximal geometric entropy in this set is given by $\trans_{\mathbf{P}_n}(\mu)$. Hence, the measures of maximal geometric entropy in $\mc M[\nu,n]$ are among the set
\[
\mc T_n = \left\{ \trans_{\mathbf{P}_n}(\mu): \mu \sim \Pi^n(\nu) \right \}.
\]
Due to the explicit entropy expression in \Cref{LEM:entropy-maximising-inverse-branch}, the maximal geometric entropy in $\mc T_n$ is obtained exactly for those measures $\mu \sim \Pi^n(\nu)$ that have maximal geometric entropy. By \Cref{THM:fixed-frequency-inverse-limit}, every such measure $\mu$ is an inverse limit measure, and hence the same holds for $\trans_{\mathbf{P}_n}(\mu)$. In summary, the measures of maximal geometric entropy in $\mc M[\nu,n]$ are $(\vartheta^n,\mc P)$ limit measures for appropriate $\mc P$. 
\end{proof}

\begin{lemma}
\label{LEM:convergence-along-inflation-words}
Let $(\mu_n)_{n \in \N}$ be a sequence of measures with $\mu_n \in \mc M[\nu,n]$ for all $n \in \N$. Then, $\lim_{n \to \infty} \mu_n = \nu$ in the weak topology.
\end{lemma}

\begin{proof}
Recall that by Kac's formula we have for an arbitrary $S$-invariant measure $\mu$ and $u \in \mc A^+$ that
\[
\mu([u]) =\frac{\mu_{A_n}(f_u)}{\mu_{A_n}(r_{A_n})}, 
\quad
f_u = \sum_{i=0}^{r_{A_n}} \chi_{[u]} \circ S^i.
\]
Assume that $\mu \in \mc M[\nu,n]$; that is $\mu_{A_n} \sim \nu_{A_n}$. Since $r_{A_n}$ is constant on cylinders of the form $[(a,v)]$ with $(a,v) \in \mc B_n$, this implies that 
$\mu_{A_n}(r_{A_n}) = \nu_{A_n}(r_{A_n})$. The same holds for the function $g_u$ which takes the constant value $|v|_u$ on $[(a,v)]$ with $(a,v) \in \mc B_n$. Since $g_u \leqslant f_u \leqslant g_u + |u|$, we obtain that
\begin{equation}
\label{EQ:mu-nu-on-u}
|\mu([u]) - \nu([u])| \leqslant \frac{|u|}{\nu_{A_n}(r_{A_n})}.
\end{equation}
Note that $r_{A_n}$ gives the (symbolic) length of the level-$n$ inflation word at the origin. Since the geometric length grows with $\lambda^n$ and the ratio between symbolic and geometric length is bounded, the difference in \eqref{EQ:mu-nu-on-u} decays exponentially with $n$.
\end{proof}

\begin{prop}
Let $\mc M_{\operatorname{I}}$ be the set of all measures $\mu$ such that $\mu$ is a $(\vartheta^n,\mc P)$ inverse limit measure for some $\mc P$ and $n \in \N$. Then, for every $\nu \in \mc M$, we can find a sequence of measures $(\mu_n)_{n \in \N}$ with $\mu_n \in \mc M_{\operatorname{I}}$ such that $\lim_{n \to \infty} \mu_n = \nu$ and 
\[
\lim_{n \to \infty} \hg_{\mu_n} = \inf_{n \in \N} \hg_{\mu_n} = \hg_\nu.
\]
In particular, $\mc M_{\operatorname{I}}$ is dense in $\mc M$.
\end{prop}

\begin{proof}
Due to \Cref{LEM:inverse-limits-for-inflations-word-frequencies}, we can choose for each $n \in \N$ an inverse limit measure $\mu_n \in \mc M[\nu,n]$ such that $h_{\mu_n} \geqslant h_{\nu}$. Since this sequence of measures mimics the inflation word frequencies of $\nu$, it converges to $\nu$ by \Cref{LEM:convergence-along-inflation-words}. The statement on convergence of the geometric entropies follows by the upper semi-continuity of entropy.
\end{proof}

This result shows that all measures of maximal geometric entropy are limits of inverse limit measures of maximal geometric entropy. In particular, if there are only finitely many inverse limit measures of maximal geometric entropy, there can be no further measures of maximal geometric entropy.
In our quest to show intrinsic ergodicity we can hence restrict our attention to inverse limit measures.

\subsection{Uniqueness of inverse limit measures}
\label{SEC:inverse-limit-uniqueness}

\textit{A priori} it is not clear whether a given sequence $\mc P$ of probability choices admits just one or several inverse limit measures. In this section, we characterise the uniqueness of inverse limit measures via the ergodicity of an associated (inverse-time) Markov chain.

Given $\mc P = (\mathbf{P}_n)_{n \in \N}$, we call $Q(\mc P) = (Q(\mathbf{P}_n))_{n \in \N}$ the \emph{$\mc P$-Markov sequence}. This represents an inverse-time inhomogenous Markov process that controls the flow of interval proportions induced by $\mc P$. More precisely, if $(\mu_n)_{n \in \N}$ is $(\vartheta,\mc P)$-adapted, then $\mu_n = \trans_{\mathbf{P_n}}(\mu_{n+1})$ implies via \Cref{COR:interval-proportion-transformation} that
\[
\pi^{\mu_n} = Q(\mathbf{P}_n) \pi^{\mu_{n+1}}
\]
for all $n \in \N$. This has a unique solution precisely if the $\mc P$-Markov sequence is ergodic.

Before we continue, let us expand a bit more on the role of $Q(\mathbf{P})$ as updating the interval proportion vectors of  periodic measure representations under the action of $\vartheta_{\mathbf{P}}$. This gives a natural analogue to \Cref{COR:interval-proportion-transformation}.

\begin{lemma}
\label{LEM:interval-proportions-under-RS}
Given a random word $\omega$ and a random substitution $\vartheta_{\mathbf{P}}$, we have that
\[
\pi^{\overline{\mu}_{\vartheta_{\mathbf{P}}(\omega)}}
= Q(\mathbf{P}) \pi^{\overline{\mu}_{\omega}}.
\]
\end{lemma}

\begin{proof}
    Let us set $M = M(\mathbf{P})$ and $Q= Q(\mathbf{P})$. 
    First, we note that by the definition of $\overline{\mu}$, we have that
    \[
    R^{\overline{\mu}_{\omega}} = \frac{\mathbb{E}[\phi(\omega)]}{\mathbb{E}[|\omega|]},
    \]
    for every random word $\omega$. Hence, for the corresponding interval proportion vector, we have
    \begin{equation}
    \label{EQ:interval-proportion-for-random-word-measure}
    \pi^{\overline{\mu}_{\omega}}_a = \frac{\Length_a \mathbb{E}[\phi(\omega)]_a}{\Length \mathbb{E}[\phi(\omega)]}.
    \end{equation}
     Note that $\mathbb{E}[\phi(\vartheta_{\mathbf{P}}(\omega))] = M \mathbb{E}[\phi(\omega)]$.
     For convenience, we use the shorthand $v = \mathbb{E}[\phi(\omega)]$ in the following.
    Applying \eqref{EQ:interval-proportion-for-random-word-measure} to the random word $\vartheta_{\mathbf{P}}(\omega)$, and recalling that $\Length_a M_{ab} = \lambda \Length_b Q_{ab}$, we obtain
    \[
    \pi^{\overline{\mu}_{\vartheta_{\mathbf{P}}(\omega)}}_a
    = \frac{\Length_a (Mv)_a}{ \Length M v}
    = \frac{\sum_{b \in \mc A} \Length_a M_{ab} v_b}{ \lambda \Length v}
    = \sum_{b \in \mc A} Q_{ab} \frac{\Length_b v_b}{\Length v}
    = (Q \pi^{\overline{\mu}_{\omega}})_a,
    \]
    as required.
\end{proof}

\begin{lemma}
\label{LEM:P-accumulations-give-measures}
    Let $\mc P$ be a sequence of probability choices with $\mc P$-Markov sequence $Q(\mc P) = (Q_n)_{n \in \N}$.
    Given $a \in \mc A$ and $m \in \N$, assume that $\pi$ is an accumulation point of $(Q_{[m,m+n]} e_a)_{n \in \N}$. Then, there is a $(\vartheta,\mc P)$-adapted sequence of measures $(\mu_n)_{n \in \N}$ such that $\pi^{\mu_m} = \pi$.
\end{lemma}

\begin{proof}
    Let $(n_j)_{j \in \N}$ be a strictly increasing subsequence such that $\pi = \lim_{j \to \infty} Q_{[m,n_j]} e_a$. Up to choosing a further subsequence, we can assume that the sequence of probability vectors $(Q_{[n_j+1,n_{2j}]} e_a)$ converges to some vector $\pi'$. We write $k_j = n_{2j} - n_j$ and note that $Q_{[n_j+1,n_{2j}]}$ is the geometric substitution matrix of the probability choice
    \[
    \mathbf{P}^{[j]}:= \mathbf{P}_{n_j+1} \cdots \mathbf{P}_{n_{2j}}.
    \]
    We consider the sequence of random words $(\omega_j)_{j \in \N}$ with $\omega_j = \vartheta^{k_j}_{\mathbf{P}^{[j]}}(a)$ for all $j \in \N$. Again up to restricting to a subsequence, we can assume that this sequence has a limit measure $\mu = \lim_{j \to \infty} \overline{\mu}_{\omega_j}$. By \Cref{LEM:interval-proportions-under-RS}, the interval proportion vector of $\overline{\mu}_{\omega_j}$ is given by
    \[
    \pi^{\overline{\mu}_{\omega_j}} = Q(\mathbf{P}^{[j]}) \pi^{\overline{\mu}(a)} = Q_{[n_j+1,n_{2j}]} e_a \xrightarrow{j \to \infty} \pi',
    \]
    implying that $\pi^\mu = \pi'$. Since all realisations of $\omega_j$ are legal words, we further know that $\mu$ is supported on $X_{\vartheta}$. 
    Let $\mu_1$ be an accumulation point of the sequence $(\mu_1^{(j)})_{j \in \N}$ with
    \[
    \mu^{(j)}_1 = \trans_{\mathbf{P}_1} \circ \ldots \circ \trans_{\mathbf{P}_{n_j}}(\mu).
    \]
    We claim that $(\Pi^n(\mu_1))_{n \in \N_0}$ is a $(\vartheta,\mc P)$-adapted sequence.
    Indeed, for $n_j \geqslant k$, we have that 
    \[
    \mu_k^{(j)} := \Pi^{k-1}(\mu_1^{(j)}) = \trans_{\mathbf{P}_k} \circ \ldots \circ \trans_{\mathbf{P}_{n_j}}(\mu)
    \]
    converges along the same subsequence to some $\mu_k$ and satisfies 
    \[
    \mu_1^{(j)} = \trans_{\mathbf{P}_1} \circ \ldots \circ \trans_{\mathbf{P}_{k-1}}(\mu_k^{(j)}),
    \]
    which persists in the limit along the corresponding subsequence. Hence $(\mu_k)_{k \in \N}$ is indeed $(\vartheta,\mc P)$-adapted. 
    For the interval proportions of $\mu_m^{(j)}$, we obtain
    \[
    \pi^{\mu_m^{(j)}} = Q_{[m,n_j]} \pi^{\mu} = Q_{[m,n_j]} \pi'.
    \]
    Since $\pi'= \lim_{j \to \infty} Q_{[n_j+1,n_{2j}]} e_a$ and the norm of $Q_{[m,n_j]}$ is uniformly bounded we obtain that
    \[
    \pi^{\mu_m} = \lim_{j \to \infty} Q_{[m,n_j]}\pi' = \lim_{j \to \infty} Q_{[m,n_{2j}]} e_a = \pi,
    \]
    as claimed.
\end{proof}

\begin{prop}
\label{PROP:unique-inverse-limit-ergodic}
 There is a unique $(\vartheta,\mc P)$ inverse limit measure if and only if $Q(\mc P)$ is ergodic.
\end{prop}

\begin{proof}
First, assume that $Q(\mc P) = (Q_n)_{n \in \N}$ is ergodic. This means that for all $n \in \N$ there is some $\pi^n$ such that $Q_{[n,n+k-1]}$ converges to $\pi^n \mathds{1}^T$ as $k \to \infty$.
Let $(\mu_n)_{n \in \N}$ be a $(\vartheta,\mc P)$-adapted sequence. By \Cref{COR:interval-proportion-transformation}, we have that
\[
\pi^{\mu_{n}} = Q_{[n,n+k-1]} \pi^{\mu_{n+k}} \xrightarrow{k \to \infty} \pi^n.
\]
Let $\mathbf{P}^{(n)} := \mathbf{P}_1 \cdots \mathbf{P}_n$. Then, $\mu = \mu_1$ satisfies that $\mu = \trans_{\mathbf{P}^{(n)}}(\mu_{n+1})$, and therefore the corresponding induced measure on $A_n = \vartheta^n(X_{\vartheta})$ satisfies
\begin{equation}
\label{EQ:induced-fixed-by-sequence}
\mu_{A_n}[(a,v)] = \mu_{n+1}([a]) \mathbf{P}^{(n)}_{v,a},
\end{equation}
for all $a \in \mc A$ and $v \in \vartheta^n(a)$.
If $(\mu_n')_{n \in \N}$ is another $(\vartheta,\mc P)$-adapted sequence with $\mu'=\mu'_1$, we have that
\[
\pi^{\mu_{n}} = \pi^n = \pi^{\mu'_n},
\]
and therefore $R^{\mu_n} = R^{\mu'_n}$ for all $n \in \N$. That is, $\mu_n\sim \mu'_n$ for all $n \in \N$, implying via \eqref{EQ:induced-fixed-by-sequence} that
$\mu_{A_n} \sim \mu'_{A_n}$. Phrased differently, we have that $\mu' \in \mc M[\mu,n]$ for all $n \in \N$.
By \Cref{LEM:convergence-along-inflation-words}, the constant sequence $(\mu')$ therefore converges to $\mu$, meaning that $\mu' = \mu$. We conclude that there can be only one $(\vartheta,\mc P)$ inverse limit measure.

Conversely, assume that $Q(\mc P) = (Q_n)_{n \in \N}$ is not ergodic. By \Cref{THM:ergodicity-characterisation_and_conditions}, there exists an $m \in \N$ such that $\delta(Q_{[m,m+k]})$ does not converge to $0$ as $k \to \infty$. Since the sequence is non-increasing, there exists a $c>0$ such that $\delta(Q_{[m,m+k]}) > c$ for all $k \in \N$. In particular, we can find $a,b \in \mc A$ such that 
\[
d_V(Q_{[m,n]}e_a, Q_{[m,n]}e_b) > c >0
\]
for all $n$ in some strictly increasing sequence $(n_j)_{j \in \N}$. Up to choosing a subsequence, we may assume that both $\pi = \lim_{j \to \infty} Q_{[m,n_j]}e_a$ and $\pi'= \lim_{j \to \infty} Q_{[m,n_j]}e_b$ exist as a limit. By construction, we have that $d_V(\pi,\pi') \geqslant c > 0$ and hence $\pi \neq \pi'$. 
By \Cref{LEM:P-accumulations-give-measures}, we can find corresponding $(\vartheta,\mc P)$-adapted sequences $(\mu_n)_{n \in \N}$ and $(\mu_n')_{n \in \N}$ such that $\pi^{\mu_m} = \pi$ and $\pi^{\mu'_m} = \pi'$. Since $\pi$ and $\pi'$ are different, so are $\mu_m$ and $\mu_m'$, and ultimately $\mu_1 \neq \mu_1'$. This implies that there are several $(\vartheta,\mc P)$ inverse limit measures.
\end{proof}

As a direct consequence of \Cref{PROP:unique-inverse-limit-ergodic} and \Cref{THM:ergodicity-characterisation_and_conditions}, we obtain the following list of sufficient conditions for the uniqueness of inverse limit measures.

\begin{corollary}
\label{COR:unique-inverse-limit-conditions}
There is a unique $(\vartheta,\mc P)$ inverse limit measure for $\mc P = (\mathbf{P}_n)_{n \in \N}$ if any of the following hold.
\begin{enumerate}
\item $\vartheta$ is compatible.
\item There is a primitive matrix $M$ such that $\lim_{n \to \infty} M(\mathbf{P}_n) = M$.
\item There is a primitive matrix $Q$ such that $\lim_{n \to \infty} Q(\mathbf{P}_n) = Q$.
\item There is some $n \in \N$ such that all marginals of $\vartheta^n$ have a strictly positive substitution matrix.
\item $\vartheta$ is defined on a binary alphabet $\mc A$.
\end{enumerate}
\end{corollary}

\begin{proof}
    The first condition is a special case of the second condition. 
    The second and third condition are equivalent and it follows directly from \Cref{THM:ergodicity-characterisation_and_conditions} that they imply the ergodicity of the corresponding Markov chain.
    If the fourth condition holds, we obtain that $\delta(Q({\mathbf{P}^{(n)}}))$ is bounded away from $1$ for all probability choices $\mathbf{P}^{(n)}$ of $\vartheta^n$. Using the submultiplicativity of $\delta$, this implies $\lim_{m \to \infty} \delta(Q_n \cdots Q_{n+m}) =0$ for $Q_n = Q(\mathbf{P}_n)$ and hence we obtain ergodicity.
    Finally, we show that the fifth condition is a special case of the fourth condition. 
    If $\vartheta$ is defined on a binary alphabet, recognisability enforces some combinatorial structure. More precisely, if $a^n \in \vartheta(a)$ or $b^n \in \vartheta(b)$, geometric compatibility (with inflation factor $\lambda>1$) enforces that $n \geqslant 2$. This would imply that $a^\Z$ or $b^\Z$ is a non-recognisable element of the subshift. On the other hand, if $b^n \in \vartheta(a)$ and $a^m \in \vartheta(b)$, it must be $n>1$ or $m>1$ and we obtain the same contradiction to recognisability. Hence, we can pick $a$ such that every word in $\vartheta(a)$ contains both $a$ and $b$, and every word in $\vartheta(b)$ contains at least one $a$. It follows that all marginals of $\vartheta^2$ have a strictly positive substitution matrix.
\end{proof}

\section{Uniformity measure and intrinsic ergodicity}
\label{SEC:uniformity_intrinsic_ergodicity}

We still assume that $\vartheta$ is a primitive, recognisable and geometrically compatible random substitution. 
Based on the previous discussion, we show that uniformity measures are the $(\vartheta,\mc P)$ limit measures for some explicit $\mc P$. We start with a slight generalisation of the $n$-productivity distributions in \Cref{DEF:n-productivity}.

\begin{definition}
For $n,m \in \N$, let $\mathbf{P}^{n,m}$ denote the $n$-productivity distribution for $\vartheta^m$, that is,
\[
\mathbf{P}^{n,m}_{v,a} = \frac{\# \vartheta^n(v)}{ \sum_{u \in \vartheta^m(a)} \# \vartheta^n(u)} = \frac{\# \vartheta^n(v)}{\# \vartheta^{n+m}(a)},
\]
for all $v \in \vartheta^m(a)$. In particular, $\mathbf{P}^{0,m}$ represents the uniform distribution on each $\vartheta^m(a)$.
\end{definition}

Note that the equality $\sum_{u \in \vartheta^m(a)} \# \vartheta^n(u) = \# \vartheta^{n+m}(a)$ makes use of the disjoint set condition, implied by recognisability of $\vartheta^m$. 
For a word $u = u_1 \cdots u_r$ and $v \in \vartheta^m(u)$, there exists, due to unique realisation paths, a unique decomposition $v = v_1 \cdots v_r$ with $v_i \in \vartheta(u_i)$ for all $1\leqslant i \leqslant r$. We obtain that
\[
\mathbb{P} \bigl[ \vartheta^m_{\mathbf{P}^{n,m}}(u) = v \bigr]
= \prod_{i=1}^r \frac{\# \vartheta^n(v_i)}{\# \vartheta^{n+m}(u_i)} = \frac{\# \vartheta^n(v)}{\# \vartheta^{n+m}(u)},
\]
again using unique realisation paths in the last step.
Intuitively, uniformity measures are those that exhibit a uniform distribution of inflation words on each level. More formally, this can be formulated as follows.

\begin{prop}
\label{PROP:uniformity-measures-distributions}
$\nu \in \mc M$ is a uniformity measure if and only if $\nu \in \mc M[\mathbf{P}^{0,n}]$ for all $n \in \N$. 
\end{prop}

\begin{proof}
    Recall from \Cref{COR:Uniformity_char} that uniformity measures are precisely those such that $\nu_{A_n}$ is invariant under $\Gamma_n$ for all $n\in \N$.
    The statement $\nu \in \mc M[\mathbf{P}^{0,n}]$ can be expressed equivalently by
    \begin{equation}
    \label{EQ:measure-equality-on-inflation-words}
    \nu_{A_n}([a,v]) = \nu_{A_n}([a,u])
    \end{equation}
    for all $a \in \mc A$, $n \in \N$ and $u,v \in \vartheta^n(a)$. Each of the automorphisms $f_\alpha \in \Gamma_{a,n}$ leaves $A_n \subset \mc B_n^\Z$ invariant by construction, and acts on it by a permutation of the letters $\{(a,v): v\in \vartheta^n(a)\} \subset \mc B_n$. Hence, if $\nu$ is a uniformity measure, we obtain
    \[
    (\nu \circ f_\alpha)_{A_n}([a,v]) = \nu_{A_n} \circ f_\alpha ([a,v]) = \nu_{A_n}([a,\alpha(v)]).
    \]
    Choosing $\alpha\in \operatorname{Sym}(\vartheta^n(a))$ with $\alpha(v) = u$ reproduces \eqref{EQ:measure-equality-on-inflation-words} and we conclude that $\nu \in \mc M[\mathbf{P}^{0,n}]$. 
    Conversely assume that $ \nu \in \mc M[\mathbf{P}^{0,n}]$ for all $n \in \N$. Let $n\in \N$ and $f \in \Gamma_n$. Since the groups $\Gamma_n$ are nested, we see that $f$ leaves $A_m$ invariant for all $m \geqslant n$, and it acts via a permutation $\alpha_{m,b}$ on $\{(b,v):v \in \vartheta^m(b) \}$ for every $b \in \mc A$. Hence,
    \[
    (\nu\circ f)_{A_m}([b,v]) = \nu_{A_m} ([b,\alpha_{m,b}(v)]) = \nu_{A_m}([b,v]),
    \]
    due to \eqref{EQ:measure-equality-on-inflation-words}. From this it follows that $(\nu \circ f)_{A_m} \sim \nu_{A_m}$ for all $m \geqslant n$. This property enforces $\nu = \nu \circ f$ by \Cref{LEM:convergence-along-inflation-words}, and it follows that $\nu$ is a uniformity measure.
\end{proof}

In the following, we prove several useful consistency relations satisfied by the $n$-productivity distributions.
\begin{lemma}
\label{LEM:productivity-splitting}
For all $n,k,m \in \N_0$, we have that
\[
\mathbf{P}^{n,k+m}
= \mathbf{P}^{n,k} \mathbf{P}^{n+k,m}.
\]
In particular, for all $n \in \N_0$ and $k \in \N$,
\[
\mathbf{P}^{n,k} = \mathbf{P}^{n,1} \mathbf{P}^{n+1,1} \cdots \mathbf{P}^{n+k-1,1}.
\]
\end{lemma}

\begin{proof}
 Given $a \in \mc A$, and $w \in \vartheta^{k+m}(a)$ let $v \in \vartheta^m(a)$ be the unique word with $w \in \vartheta^k(v)$. We obtain 
\begin{align*}
(\mathbf{P}^{n,k}\mathbf{P}^{n+k,m})_{w,a}
 & =  \mathbf{P}^{n,k}_{w,v} \mathbf{P}^{n+k,m}_{v,a}
= \frac{\# \vartheta^n(w)}{\# \vartheta^{n+k}(v)}
\frac{\#\vartheta^{n+k}(v)}{\# \vartheta^{n+k+m}(a)}
\\ &= \frac{\# \vartheta^n(w)}{\# \vartheta^{n+k+m}(a)}
= \mathbf{P}^{n,k+m}_{w,a},
\end{align*}
proving the first relation. Iterating this relation gives the second claim.
\end{proof}

\begin{definition}
Given $\mc P = (\mathbf{P}^{n,1} )_{n \in \N_0}$, we call every $(\vartheta,\mc P)$-adapted sequence of measures $(\mu_n)_{n \in \N_0}$ a \emph{uniformity sequence}, and every $(\vartheta,\mc P)$ inverse limit measure is referred to as a \emph{uniformity limit measure}.
\end{definition}

Recall that $\nu$ is a \emph{uniformity measure} if $\nu \in \mc M[\mathbf{P}^{0,n}]$ for all $n \in \N$, due to \Cref{PROP:uniformity-measures-distributions}. We will see that this concept coincides with that of a uniformity limit measure.
In fact, we can show that uniformity (limit) measures are precisely the measures of maximal entropy. This gives us a slight strengthening of \Cref{THM:main_uniformity_maximal}.

\begin{theorem}
\label{THM:max-geometric-entropy-is-uniform}
Let $\vartheta$ be primitive, geometrically compatible and recognisable.
An invariant probability measure on $(X_{\vartheta},S)$ has maximal geometric entropy if and only if it is a uniformity measure if and only if it is a uniformity limit measure.
\end{theorem}

\begin{proof}
If $\nu$ is a measure of maximal geometric entropy, it maximises in particular the geometric entropy in $\Pi^{-n}(\Pi^n(\nu))$. Hence, by \Cref{COR:entropy-maximising-uniform}, we have that $\nu = \trans_{\mathbf{P}^{0,n}}(\Pi^n(\nu)) \in \mc M[\mathbf{P}^{0,n}]$ for all $n \in \N$. Hence, $\nu$ is a uniformity measure.
If $\nu = \nu_0$ is a uniformity measure, we claim that $(\nu_n)_{n \in \N_0}$ with $\nu_n = \Pi^n(\nu_0)$ is a uniformity sequence. First, we note that, by \Cref{LEM:productivity-splitting},
\[
\trans_{\mathbf{P}^{0,1}} \circ \cdots \circ \trans_{\mathbf{P}^{n-1,1}}(\nu_n) =
\trans_{\mathbf{P}^{0,n}}(\nu_n) \in \mc M[\mathbf{P}^{0,n}] \cap \Pi^{-n}(\nu_n) = \mc M[\nu,n] \cap \Pi^{-n}(\nu_n)
\]
converges to $\nu$ by \Cref{LEM:convergence-along-inflation-words}. By continuity of $\Pi$, applying $\Pi^m$ to this relation
yields that 
\[
\lim_{n \to \infty} \trans_{\mathbf{P}^{m,1}} \circ \cdots \circ \trans_{\mathbf{P}^{n-1,1}}(\nu_n) = \nu_m,
\]
for all $m \in \N$. In particular,
\[
\nu_{m-1} = \lim_{n \to \infty} \trans_{\mathbf{P}^{m-1,1}} \circ \cdots \circ \trans_{\mathbf{P}^{n-1,1}}(\nu_n) = \trans_{\mathbf{P}^{m-1,1}} (\nu_m),
\]
so $(\nu_n)_{n \in \N_0}$ is $(\vartheta,\mc P)$-adapted for $\mc P = (\mathbf{P}^{n,1})_{n \in \N_0}$. Hence, $\nu_0$ is a uniformity limit measure.
Finally, let $\nu_0$ be a uniformity limit measure with uniformity sequence $(\nu_n)_{n \in \N_0}$. We obtain $\nu_0 = \trans_{\mathbf{P}^{0,n}}(\nu_n)$, and hence we can express the geometric entropy of $\nu_0$ via \Cref{COR:entropy-maximising-uniform} as
\[
\hg_{\nu_0} = \frac{1}{\lambda^n} \biggl( \hg_{\nu_n} + \frac{1}{\Length R^{\nu_n}} \sum_{a \in \mc A} \nu_n([a]) \log \# \vartheta^n(a) \biggr),
\]
for all $n \in \N$. By \Cref{THM:geom-inf-entropy}, we have that $\lambda^{-n} \log \# \vartheta^n(a)$ converges to $\Length_a h_{\operatorname{top}}(Y_{\vartheta})$ as $n \to \infty$. Hence, performing this limit in the last relation gives
\[
\hg_{\nu_0} = h_{\operatorname{top}}(Y_{\vartheta})
\]
and we conclude that $\nu_0$ has maximal geometric entropy.
\end{proof}

\begin{corollary}
    Every measure of maximal entropy on $(Y_\vartheta,T)$ has full topological support.
\end{corollary}

\begin{proof}
    For every legal word $v$ there exists some power $n \in \N$ such that $v$ is contained in some realisation of $\vartheta^n(a)$ for every $a \in \mc A$, due to primitivity. Every uniformity measure equidistributes the inflation words of level $n$ of every given type. Hence, it assigns positive mass to the cylinder $[v]$. This shows that measures of maximal geometric entropy have full support on $X_{\vartheta}$, and hence their lifts have full support on $Y_\vartheta$.
\end{proof}

We now summarise some of our main results on intrinsic ergodicity, covering in particular \Cref{THM:main_Markov_chain_condition}.

\begin{theorem}
\label{THM:main-intrinsic-ergodicity}
    Let $\vartheta$ be a primitive, geometrically compatible and recognisable random substitution.
    There is a unique uniformity measure $\mu_{\operatorname{u}}$ if and only if the Markov sequence $(Q(\mathbf{P}^{n,1}))_{n \in \N_0}$ is ergodic. In this case, $(Y_\vartheta,T)$ is intrinsically ergodic, and the measure of maximal entropy is $\widetilde{\mu_{\operatorname{u}}}$, the lift of $\mu_{\operatorname{u}}$ under the suspension.
\end{theorem}

\begin{proof}
    The first statement about the uniqueness of $\mu_{\operatorname{u}}$ is a direct consequence of \Cref{PROP:unique-inverse-limit-ergodic}. By \Cref{THM:max-geometric-entropy-is-uniform} the uniformity measures are precisely the measures of maximal geometric entropy, and hence their lifts under the suspension are precisely the measures of maximal entropy on $(Y_\vartheta,T)$.
\end{proof}

\begin{corollary}
\label{COR:intrinsically_ergodic_conditions}
    The system $(Y_\vartheta,T)$ is intrinsically ergodic if any of the following hold.
    \begin{itemize}
    \item $\vartheta$ is compatible. 
    \item There is a primitive matrix $M$ such that $\lim_{n \to \infty} M(\mathbf{P}^{n,1}) = M$.
    \item There is a primitive matrix $Q$ such that $\lim_{n \to \infty} Q(\mathbf{P}^{n,1}) = Q$.
    \item There is some $n \in \N$ such that all marginals of $\vartheta^n$ have a strictly positive substitution matrix.
    \item $\vartheta$ is defined on a binary alphabet $\mc A$.
    \end{itemize}
\end{corollary}

\begin{proof}
    This follows directly by combining \Cref{THM:main-intrinsic-ergodicity} with \Cref{COR:unique-inverse-limit-conditions}.
\end{proof}

\begin{proof}[Proof of \Cref{COR:main_rpc}]
    If $\vartheta$ is primitive, compatible and recognisable, intrinsic ergodicity of $(Y_{\vartheta}, T)$ and $(X_{\vartheta}, S)$ are equivalent and follow by \Cref{COR:intrinsically_ergodic_conditions}. In this case, all $n$-productivity distributions are given by the uniform distribution $\mathbf{P} = \mathbf{P}^{(0,1)}$, and the uniformity sequence $\mc P$ is trivial. It hence follows from \Cref{REM:frequency_measures-as-inverse-limits} that $\mu_\mathbf{P}$ is the unique uniformity (limit) measure and thus the measure of maximal entropy.
\end{proof}

\begin{example}\label{Exa:IE-MME-not-freq}
Consider the random substitution $\vartheta$ on $\mc A = \{a,b,c\}$, given by 
\[
\vartheta \colon a \mapsto \{abc, acc\}, \quad b \mapsto \{bac,bcc\}, \quad c \mapsto \{ aac\}.
\]
This example is easily verified to be primitive, of constant length and recognisable. We will show that it gives rise to an intrinsically ergodic subshift, although the productivity weights are non-trivial. First, note that $p_n := \# \vartheta^n(a) = \# \vartheta^n(b)$ follows by induction. Similarly, let $q_n = \# \vartheta^n(c)$ and $r_n = p_n/q_n$ for all $n \in \N_0$. We obtain 
\[
q_{n+1} = \# \vartheta^{n+1}(c) = \# \vartheta^n(aac) = p_n^2 q_n,
\]
and similarly $p_{n+1} = (p_n+q_n) p_n q_n$. This yields the recursive relation $r_{n+1} = 1 + 1/r_n$ with $r_0 = 1$, which is solved by $r_n = 1 + F_n/F_{n+1}$, with $F_n$ being the $n$th Fibonacci number. Hence, the limiting value $\tau = \lim_{n \to \infty} r_n$ is the inverse of the golden ratio. From this, we obtain that $ \mathbf{P} = \lim_{n \to \infty} \mathbf{P}^{n,1}$ exists and is non-degenerate. In particular $Q(\mathbf{P}^{n,1})$ converges to the primitive matrix $Q(\mathbf{P})$. As a consequence of \Cref{COR:intrinsically_ergodic_conditions} we see that both $(X_\vartheta,S)$ and $(Y_\vartheta,T)$ are intrinsically ergodic.
Note that if we replace $c \mapsto \{ aac\}$ by $c \mapsto \{ acc\}$, we obtain $r_{n+1} = r_n + 1$ such that $r_n \to \infty$ as $n \to \infty$ and hence $\mathbf{P}$ turns out to be degenerate for $\vartheta$. In fact, it singles out the marginal $a \mapsto abc$, $b \mapsto bac$, $c \mapsto aac$. Since this is still primitive, we again deduce intrinsic ergodicity by the same criterion.
\end{example}

Easy sufficient conditions for the violation of intrinsic ergodicity seem to be harder to find. However, we provide an example below, showing that there are indeed primitive, geometrically and recognisable random substitutions with multiple measures of maximal (geometric) entropy.

\begin{example}
    Consider the primitive random substitution $\vartheta$ of constant length $4$ on the alphabet $\mc A = \{a_0,a_1,b_0,b_1,c \}$ given by 
    \[
    \vartheta \colon \begin{cases}
    a_i  & \mapsto\begin{cases}a_i a_i a_{i+1} a_{i+1}, 
    \\a_i a_{i+1} c c, \end{cases}
    \\b_i & \mapsto \begin{cases} b_i b_i b_{i+1} b_{i+1},
    \\ b_i b_{i+1} c c, \end{cases}
    \\c  & \mapsto a_0 b_0 c c,
    \end{cases}
    \]
where indices are to be understood modulo $2$.
It is straightforward to see that $\vartheta$ is primitive. To verify recognisability, note that inflation words that contain $cc$ are easy to identify. A pattern of the form $w = a_i a_i a_{i+1} a_{i+1}$ is either a complete inflation word or splits into two inflation words in the middle. The only case in which the next four letters do not force one of the two options is if they form exactly the same word $w$. Repeating the argument, we see that the only obstruction to recognisability would be the existence of words of the form $w^n$ for arbitrarily large $n$. However, $w^6$ cannot be legal, as this would require a word $a_i^6$ or $a_{i+1}^5$ in the preimage, both of which are not legal. By symmetry, the same argument applies to patterns of the form $b_i b_i b_{i+1} b_{i+1}$, and we obtain that $\vartheta$ is indeed recognisable.

The idea behind this example is the following. We can partition the alphabet into three pieces, according to $\mc A = \{a_0,a_1 \} \cup \{b_0,b_1 \} \cup \{c\}$. The letter $c$ ensures primitivity but contributes least to entropy production. For the $n$-productivity weights, this causes the images of letters of type $a$ to favour those inflation words that consist only of type $a$ letters. The same holds for letters of type $b$. In the limit, this creates a non-primitive substitution matrix. We verify that the communication to letters of a different type dies out sufficiently fast so that most of the mass starting on $a$ (or $b$) remains trapped. This precludes convergence to a common limit distribution. The details follow.

We can show by induction on $n \in \N_0$ that $p_n=\# \vartheta^n(a_i)=\# \vartheta^n(b_i)$ does not depend on $i$. We also use the notation $q_n := \# \vartheta^n(c)$ and $r_n = p_n/q_n$ for all $n \in \N_0$. Since the disjoint set condition holds, we obtain $p_{n+1} = p_n^4 + p_n^2 q_n^2$ and $q_{n+1} = p_n^2 q_n^2$. This yields
\[
r_{n+1} = r_n^2 + 1,
\]
for all $n \in \N_0$, with $r_0=1$. This is a rapidly increasing function in $n$. Let $u^a_i = a_i a_i a_{i+1} a_{i+1}$ and $v^a_i = a_{i} a_{i+1} c c$, and define $u^b_i,v^b_i$ analogously. The cardinalities satisfy
\[
\frac{\# \vartheta^n(u^a_i)}{\# \vartheta^n(v^a_i)} = \frac{p_n^2}{q_n^2} = r_n^2
\]
The $n$-productivity distribution for $\vartheta$ therefore satisfies
\[
\mathbf{P}^{n,1}_{v^a_i,a_i} =
\mathbf{P}^{n,1}_{v^b_i,b_i} = 
\frac{\# \vartheta^n(v^a_i)}{\# \vartheta^n(v^a_i) + \# \vartheta^n(u^a_i)} = \frac{1}{1+r_n^2} = \frac{1}{r_{n+1}}.
\]
Note that $\vartheta$ is a mixture of the marginals $\theta$ and $\theta'$, where 
\[
\theta \colon a_i \mapsto u^a_i,
\quad b_i \mapsto u^b_i,
\quad c \mapsto a_0 b_0 c c,
\qquad
\theta' \colon a_i \mapsto v^a_i,
\quad b_i \mapsto v^b_i,
\quad c \mapsto a_0 b_0 c c,
\]
which have substitution matrices
\[
M = \begin{pmatrix}
      2 & 2 & 0 & 0 & 1
    \\2 & 2 & 0 & 0 & 0
    \\0 & 0 & 2 & 2 & 1
    \\0 & 0 & 2 & 2 & 0
    \\0 & 0 & 0 & 0 & 2
\end{pmatrix},
\qquad
M' = \begin{pmatrix}
      1 & 1 & 0 & 0 & 1
    \\1 & 1 & 0 & 0 & 0
    \\0 & 0 & 1 & 1 & 1
    \\0 & 0 & 1 & 1 & 0
    \\2 & 2 & 2 & 2 & 2
\end{pmatrix},
\]
respectively. 
Note that in the limit $n \to \infty$ the $n$-productivity weights single out the marginal $\theta$, and hence the limiting productivity matrix $Q = \lim_{n \to \infty} Q(\mathbf{P}^{n,1})$ is given by the normalised substitution matrix $Q = M/4$. The substitution matrix for $\mathbf{P}^{n,1}$ is given by 
\[
M(\mathbf{P}^{n,1}) = \frac{r_{n+1} - 1}{r_{n+1}} M + \frac{1}{r_{n+1}} M',
\]
and the corresponding geometric variant is $ Q_n := Q(\mathbf{P}^{n,1}) = M(\mathbf{P}^{n,1})/4$.

Our aim in the following is to rule out intrinsic ergodicity by showing explicitly that $Q_{[1,n]} = Q_1 \cdots Q_n$ does not converge to a one-dimensional projection. To this end, we extract from any matrix $P$ indexed by $\mc A$ the ``upper left corner'' via
\[
A(P) = \{P_{a_i,a_j} \}_{i,j \in \{0,1 \}}.
\]
Since both $M$ and $M'$ exhibit a multiple of the idempotent matrix
\[
N = \begin{pmatrix}
    1/2 & 1/2
    \\1/2 & 1/2
\end{pmatrix}
\]
at the corresponding position, we obtain that
\[
A(Q_n) = s_n N,
\quad 
s_n := \frac{r_{n+1}-1}{r_{n+1}} + \frac{1}{2 r_{n+1}}
= \frac{2r_{n+1} - 1}{2 r_{n+1}}.
\]
Since we are dealing with non-negative matrices, extracting a submatrix is super-multiplicative in the sense that $A(QQ') \geqslant A(Q) A(Q')$, hence
\begin{equation}
\label{EQ:upper-corner-estimate}
A(Q_1 \cdots Q_n)
\geqslant A(Q_1) \cdots A(Q_n) 
= \prod_{i=0}^{n-1} s_i N.
\end{equation}
For $n \in \N_0$, we compute the first few values of $r_{n+1}$ as $2,5,26,677,458330,\ldots$, giving rise to the values $s_n = 3/4, 9/10, 51/52, \ldots$.
We argue that $s^{\infty} = \prod_{i=0}^{\infty} s_i > 1/2$. 
For every $n,m\in \N$, we can iterate the relation $r_{n+1} > r_n^2$ to obtain $r_{n+m} > r_n^{2^m} > r_n^{m+1}$,
and using that
\[
\log (s_n) \geqslant 2(s_n - 1) = - \frac{1}{r_{n+1}}
\]
for $s_n < 1$ sufficiently close to $1$, we obtain that
\[
\log \prod_{i=n}^\infty s_i
= \sum_{i = 0}^\infty \log(s_{n+i})
\geqslant - \sum_{i=0}^\infty \frac{1}{r_{n+1+i}}
> - \sum_{i=0}^\infty \frac{1}{r_{n+1}^{i+1}}
= \frac{1}{1-r_{n+1}},
\]
and therefore
\[
\prod_{i=n}^{\infty} s_i \geqslant \exp(1/(1-r_{n+1})) =: t_n.
\]
By explicit calculation, we obtain $t_2>0.96$ and therefore
\[
s^{\infty} = s_0 s_1 \prod_{i=2}^{\infty} s_i
> \frac{3}{4} \, \frac{9}{10} \, \frac{96}{100} 
= \frac{648}{1000} > 0.5 .
\]
Combining this with \eqref{EQ:upper-corner-estimate}, we obtain that for all $n \in \N$, we have
\[
(Q_{[1,n]})_{a_0,a_0} + (Q_{[1,n]})_{a_1,a_0} 
\geqslant s^\infty > 0.5.
\]
By symmetry, the same relation holds with $a$ replaced by $b$. But this means that the first and third column of $Q_{[1,n]}$ stay bounded away from each other by a positive distance. Hence, $Q_{[1,n]}$ does not converge to a one-dimensional projection. We conclude that $(Q_n)_{n \in \N}$ is not ergodic and hence that both $(X_{\vartheta},S)$ and $(Y_{\vartheta},T)$ admit several measures of maximal entropy.
\end{example}

\section*{Acknowledgements}

AM thanks Lund University for their hospitality during a research visit in April 2024, where this project began.
PG acknowledges support from the German Research Foundation (DFG) through Project 50942770 and AM acknowledges support from EPSRC grant EP/Y023358/1 and an EPSRC Doctoral Prize Fellowship.

\subsection*{Data availability statement}

This project has no associated data.

\bibliographystyle{abbrv}
\bibliography{ref}

\end{document}